\documentclass[a4paper,11pt]{amsart}

\usepackage[a4paper,hmargin={2.5cm,2.5cm},vmargin={2.5cm,2.5cm}]{geometry}

\usepackage[leqno]{amsmath}
\usepackage{amssymb,amsthm}
\usepackage[mathscr]{euscript}
\usepackage{bbm}
\usepackage{dsfont}
\usepackage[all,arc]{xy}
\usepackage{txfonts}
\DeclareMathAlphabet{\mathpzc}{OT1}{pzc}{m}{it}

\theoremstyle{plain}
\newtheorem{theorem}{Theorem}[section]
\newtheorem{lemma}[theorem]{Lemma}
\newtheorem{proposition}[theorem]{Proposition}
\newtheorem{corollary}[theorem]{Corollary}

\theoremstyle{definition}
\newtheorem{definition}[theorem]{Definition}
\newtheorem{examples}[theorem]{Examples}
\newtheorem{example}[theorem]{Example}

\theoremstyle{remark}
\newtheorem{remark}[theorem]{Remark}
\newtheorem{remarks}[theorem]{Remarks}

\newenvironment{eqcond}{\begin{enumerate}}{\end{enumerate}}
\renewcommand{\theenumii}{\alph{enumii}}

\newcommand{\Rw}{\Rightarrow}

\newcommand{\hrw}{\hookrightarrow}

\newcommand{\fra}{\mathfrak{a}}

\newcommand{\frf}{\mathfrak{f}}
\newcommand{\frg}{\mathfrak{g}}

\newcommand{\frp}{\mathfrak{p}}
\newcommand{\frq}{\mathfrak{q}}

\newcommand{\frw}{\mathfrak{w}}
\newcommand{\frv}{\mathfrak{v}}
\newcommand{\frx}{\mathfrak{x}}
\newcommand{\fry}{\mathfrak{y}}
\newcommand{\frz}{\mathfrak{z}}

\newcommand{\calA}{\mathcal{A}}
\newcommand{\calB}{\mathcal{B}}

\newcommand{\frW}{\mathfrak{W}}
\newcommand{\frX}{\mathfrak{X}}

\DeclareMathOperator{\ev}{ev}
\DeclareMathOperator{\Id}{Id}
\DeclareMathOperator{\cl}{cl}

\DeclareMathOperator{\ForgetSET}{O}
\DeclareMathOperator{\ForgetToV}{S}
\DeclareMathOperator{\ForgetToVAd}{A}
\DeclareMathOperator{\MFunctor}{M}
\DeclareMathOperator{\can}{can}
\DeclareMathOperator{\colim}{colim}
\DeclareMathOperator{\yoneda}{\mathpzc{y}}
\DeclareMathOperator{\Sup}{Sup}
\DeclareMathOperator{\upc}{\uparrow\!}

\newcommand{\mate}[1]{\,^\ulcorner\! #1^\urcorner}

\newcommand{\fspstr}[2]{\llbracket #1,#2\rrbracket}

\newcommand{\catfont}[1]{\mathsf{#1}}

\newcommand{\V}{\catfont{V}}

\newcommand{\two}{\catfont{2}}

\newcommand{\quantale}{(\V,\otimes,k)}

\newcommand{\Pplus}{\catfont{P}_{\!\!{_+}}}

\newcommand{\SET}{\catfont{Set}}

\newcommand{\TOP}{\catfont{Top}}
\newcommand{\AP}{\catfont{App}}

\newcommand{\ORD}{\catfont{Ord}}

\newcommand{\Mat}[1]{#1\text{-}\catfont{Rel}}
\newcommand{\Mod}[1]{#1\text{-}\catfont{Mod}}

\newcommand{\UMat}[1]{#1\text{-}\catfont{URel}}
\newcommand{\Cat}[1]{#1\text{-}\catfont{Cat}}
\newcommand{\Gph}[1]{#1\text{-}\catfont{Gph}}
\newcommand{\CatSep}[1]{#1\text{-}\catfont{Cat}_\mathrm{sep}}
\newcommand{\CatCompl}[1]{#1\text{-}\catfont{Cat}_\mathrm{cpl}}
\newcommand{\Cocompl}[1]{#1\text{-}\catfont{Cocont}}
\newcommand{\CocomplSep}[1]{#1\text{-}\catfont{Cocont}_\mathrm{sep}}
\newcommand{\ICocompl}[1]{#1\text{-}\catfont{ICocont}}
\newcommand{\ICocomplSep}[1]{#1\text{-}\catfont{ICocont}_\mathrm{sep}}

\renewcommand{\to}{\longrightarrow}
\renewcommand{\mapsto}{\longmapsto}
\newcommand{\relto}{{\longrightarrow\hspace*{-2.8ex}{\mapstochar}\hspace*{2.6ex}}}
\newcommand{\modto}{{\longrightarrow\hspace*{-2.8ex}{\circ}\hspace*{1.2ex}}}
\newcommand{\kto}{\relbar\joinrel\rightharpoonup}
\newcommand{\krelto}{\,{\kto\hspace*{-2.5ex}{\mapstochar}\hspace*{2.6ex}}}
\newcommand{\kmodto}{\,{\kto\hspace*{-2.8ex}{\circ}\hspace*{1.3ex}}}

\newcommand{\kleisli}{\circ}

\newcommand{\homkleisliright}{\multimapinv}

\newcommand{\homcompleft}{\multimapdot}
\newcommand{\homcompright}{\multimapdotinv}

\newcommand{\Txi}{T_{\!\!_\xi}}

\newcommand{\monadfont}[1]{\mathbbm{#1}}

\newcommand{\mT}{\monadfont{T}}
\newcommand{\mI}{\monadfont{1}}
\newcommand{\mU}{\monadfont{U}}
\newcommand{\mL}{\monadfont{L}}
\newcommand{\mInj}{\monadfont{I}}
\newcommand{\mInjP}{\monadfont{I}^+}

\newcommand{\mS}{\monadfont{S}}

\newcommand{\monad}{(T,e,m)}

\newcommand{\imonad}{(\Id,1,1)}
\newcommand{\umonad}{(U,e,m)}
\newcommand{\wmonad}{(L,e,m)}

\newcommand{\theoryfont}[1]{\mathscr{#1}}

\newcommand{\Tth}{\theoryfont{T}}
\newcommand{\Ith}{\theoryfont{I}}
\newcommand{\Uth}{\theoryfont{U}}

\newcommand{\toptheory}{(\mT,\V,\xi)}
\newcommand{\itheory}{(\mI,\V,1_{\V})}

\newcommand{\sptheory}{(\mU,\two,\xi_{\two})}

\newcommand{\BC}{(BC)}

\newcommand{\doo}[1]{\overset{\centerdot}{#1}}

\newcommand{\op}{\mathrm{op}}

\usepackage[a4paper,hypertexnames=false]{hyperref} %plainpages

\begin{document}

\title{Injective Spaces via Adjunction}

\author{Dirk Hofmann}
\thanks{The author acknowledges partial financial assistance by Unidade de Investiga\c{c}\~{a}o e Desenvolvimento Matem\'{a}tica e Aplica\c{c}\~{o}es da Universidade de Aveiro/FCT}
\address{Departamento de Matem\'{a}tica\\ Universidade de Aveiro\\3810-193 Aveiro\\ Portugal}
\email{dirk@ua.pt}

\subjclass[2000]{18A05, 18D15, 18D20, 18B35, 18C15, 54B30, 54A20}
\keywords{Quantale, $\V$-category, monad, topological theory, module, Yoneda lemma, weighted colimit}

\begin{abstract}
Our work over the past years shows that not only the collection of (for instance) all topological spaces gives rise to a category, but also each topological space can be seen individually as a category by interpreting the convergence relation $\mathfrak{x}\to x$ between ultrafilters and points of a topological space $X$ as arrows in $X$. Naturally, this point of view opens the door to the use of concepts and ideas from (enriched) Category Theory for the investigation of (for instance) topological spaces. In this paper we study cocompleteness, adjoint functors and Kan extensions in the context of topological theories. We show that the cocomplete spaces are precisely the injective spaces, and they are algebras for a suitable monad on $\SET$. This way we obtain enriched versions of known results about injective topological spaces and continuous lattices.
\end{abstract}

\maketitle

\section*{Introduction}
% It is the aim of this paper to provide a first step towards this direction. First of all, we have to agree on what means ``suitable category'' here. To explain our idea, recall that
% \begin{enumerate}
% \item ordered sets can be viewed as $\V$-enriched categories where the quantale $\V$ is given by the two-element Boolean algebra $\two$. Furthermore, (generalised) metric spaces coincides with categories enriched over the quantale $\V=\Pplus$ where $\Pplus=([0,\infty]^\op,+,0)$.
% \item topological spaces can be described as $(\mU,\two)$-categories where $\mU$ denotes the ultrafilter monad, whereby $(\mU,\Pplus)$-categories are precisely the approach spaces.
% \end{enumerate}
% The conclusion is now very simple: a continuous metric space ``should be'' an injective approach space, and, more general, a continuous $\V$-categories ``should be'' an injective $(\mU,\V)$-category. Consequentely, we are now left with the problem of studying injective $(\mU,\V)$-categories and their relation with $\V$-categories.

The title of the present article is clearly reminiscent of the chapter \emph{Ordered sets via adjunctions} by R.\ Wood \cite{Woo_OrdAdj}, where the theory of ordered sets is developed elegantly  employing consequently the concept of adjunction. One of the fundamental aspects of our recent research is described by the slogan \emph{topological spaces are categories}, and therefore can be studied using notions and techniques from (enriched) Category Theory. We hope to be able to show in this paper that concepts like module, colimit and adjointness can be a very useful tool for the study of topological spaces too.

We should explain what is meant by ``spaces are categories''. In his famous 1973 paper \cite{Law_MetLogClo} F.W.\ Lawvere considers the points of a (generalised) metric space $X$ as the objects of a category $X$ and lets the distance
\[
d(x,y)\in[0,\infty]
\]
play the role of the hom-set of $x$ and $y$. In fact, the basic laws
\begin{align*}
0\ge d(x,x) &&\text{and} && d(x,y)+d(y,z)\ge d(x,z)
\end{align*}
remind us immediately to the operations ``choosing the identity'' and ``composition''
\begin{align*}
1\to\hom(x,x) &&\text{and} && \hom(x,y)\times\hom(y,z)\to\hom(x,z)
\end{align*}
of a category. Motivated by Lawvere's approach, we consider the points of a topological space $X$ as the objects of our category, and interprete the convergence $\mathfrak{x}\to x$ of an ultrafilter $\frx$ on $X$ to a point $x\in X$ as a morphism in $X$. With this interpretation, the convergence relation
\begin{equation}\label{homfunct}
 \to:UX\times X\to\two
\end{equation}
becomes the ``hom-functor'' of $X$. Clearly, we have to make here the concession that a morphism in $X$ does not have just an object but rather an ultrafilter (of objects) as domain. This intuition is supported by the observation (due to M.\ Barr \cite{Bar_RelAlg}) that a relation $\mathfrak{x}\to x$ between ultrafilters and points of a set $X$ is the convergence relation of a (unique) topology on $X$ if and only if
\begin{align}\label{TCatAx}
e_X(x)\to x &&\text{and}&& (\mathfrak{X}\to\mathfrak{x}\;\&\;\mathfrak{x}\to x)\,\models\, m_X(\mathfrak{X})\to x,
\end{align}
for all $x\in X$, $\mathfrak{x}\in UX$ and $\mathfrak{X}\in UUX$, where $m_X(\mathfrak{X})$ is the filtered sum of the filters in $\mathfrak{X}$ and $e_X(x)=\doo{x}$ the principal ultrafilter generated by $x\in X$. In the second axiom we use the natural extension of a relation between ultrafilters and points to a relation between ultrafilters of ultrafilters and ultrafilters, so that $\mathfrak{X}\to\mathfrak{x}$ is a meaningful expression. In our interpretation, the first condition postulates the existence of an ``identity arrow'' on $X$, whereby the second one requires the existence of a ``composite'' of ``composable pairs of arrows''. Furthermore, a function $f:X\to Y$ between topological spaces is continuous whenever $\frx\to x$ in $X$ implies $f(\frx)\to f(x)$ in $Y$, that is, $f$ associates to each object in $X$ an object in $Y$ and to each arrow in $X$ an arrow in $Y$ between the corresponding (ultrafilter of) objects in $Y$. It is now a little step to admit that the hom-functor \eqref{homfunct} of such a category $X$ takes values in a quantale $\V$ other than the two-element Boolean algebra $\two$, and that the domain $\frx$ of an arrow $\frx\to x$ in $X$ is an element of a set $TX$ other than the set $UX$ of all ultrafilters of $X$. As one can see immediately, we need $T$ to be a functor $T:\SET\to\SET$ in order to define the notion of functor between such categories, moreover, we need $T$ to be part of a $\SET$-monad $\mT=\monad$ in order to formulate the axioms \eqref{TCatAx} of a category in this context. Eventually, we reach the notion of a $(\mT,\V)$-category (also called $(\mT,\V)$-algebra or lax algebras), for a $\SET$-monad $\mT$ and quantale $\V$, as introduced in \cite{CH_TopFeat,CT_MultiCat,CHT_OneSetting}. A different but related approach to this kind of categories was presented by Burroni \cite{Bur_TCat}. 

Though the initial paper \cite{CH_TopFeat} focused on the topological features of this approach, already in \cite{CT_MultiCat} the emphasis was put on the categorical description of $(\mT,\V)$-algebras. The theory of categories enriched in a monoidal closed category $\V$ is by now classical \cite{Ben:63,Ben:65,EK_CloCat,Kel_EnrCat,Law_MetLogClo}. We have a wide range of concepts and theorems at our disposal, it includes such things as modules (also called distributors, profunctors), weighted (co)limits, the Yoneda Lemma, Kan extensions, adjoint functors, and many more. Naturally, we wish to lift these notions and results to the $(\mT,\V)$-setting. A first step in this direction was done in \cite{CH_Compl}, where the notion of module is introduced into the realm of (now called) $(\mT,\V)$-categories. As in the case of $\V$-categories, this concept is fundamental for the further development of the theory; for instance, completeness properties of $(\mT,\V)$-categories are formulated in terms of modules. In fact, in \cite{CH_Compl} the categorical notion of Cauchy-completeness (the name Lawvere-completeness respectively L-completeness is proposed in \cite{CH_Compl,HT_LCls}) is introduced and studied. A further achievement of \cite{CH_Compl} is the formulation and proof of a $(\mT,\V)$-version of the famous Yoneda lemma, a result which turns out to be crucial for the study of $(\mT,\V)$-categories in the same way as the classical result is for the development of the theory of $\V$-categories. This can be judged by looking at the results and proofs of the subsequent paper \cite{HT_LCls} and also the present one. However, in order to proceed with our ``spaces as categories'' project, further conditions on the monad $\mT$ and the quantale $\V$ are needed. As a result of our work on this subject emerged the notion of a \emph{topological theory} $\Tth=\toptheory$ introduced in \cite{Hof_TopTh}, where we add a map $\xi:T\V\to\V$ compatible with the monad and the quantale structure to our setting. Our experience shows so far that this concept is broad enough to include our principal examples, and at the same time restrictive enough to allow us to introduce categorical ideas into the realm of $(\mT,\V)$-categories (which we now call $\Tth$-categories).

The particular topic of this paper is the study of weighted colimits, cocomplete $\Tth$-categories and adjoint $\Tth$-functors. We start by recalling the definition of the principal players, namely $\Tth$-categories, $\Tth$-functors and $\Tth$-modules, and then proceed introducing adjoint $\Tth$-functors and weighted colimits for $\Tth$-categories precisely as for $\V$-categories. Furthermore, we show that the development of many basic properties does not go much beyond the $\V$-category case, as soon as we have $\Tth$-substitutes for dual category, presheaf-construction and the Yoneda lemma available. Finding useful equivalents to these notions and results we see as one of the main challenges here, fortunately, most of these problems are already solved in \cite{CH_Compl}. However, in this paper we give a different approach to the Yoneda lemma, by proving a more general result (Theorem \ref{Yoneda}) more suitable for our purpose. Moreover, our proof does not need anymore the restrictive condition $T1=1$. The achievements of this paper can then be summarised as follows. We characterise cocomplete $\Tth$-categories as precisely the injective ones with respect to fully faithful $\Tth$-functors, and as those $\Tth$-categories $X$ for which the Yoneda functor $\yoneda_X:X\to\hat{X}$ into the presheaf $\Tth$-category $\hat{X}$ has a left adjoint. We deduce cocompleteness of the presheaf $\Tth$-category $\hat{X}$, and show the existence of Kan-extensions in our setting, that is, any $\Tth$-functor $f:X\to Y$ into a cocomplete $\Tth$-category has an (up to equivalence) unique extension to a left adjoint $\Tth$-functor $f_L:\hat{X}\to Y$. As a consequence, we see that the category $\CocomplSep{\Tth}$ of separated and cocomplete (=injective) $\Tth$-categories and left adjoint $\Tth$-functors is a reflective subcategory of $\Cat{\Tth}$ (and of $\CatSep{\Tth}$), the category of (separated) $\Tth$-categories and $\Tth$-functors. Furthermore, we show that the induced monad on $\CatSep{\Tth}$ is of Kock-Z\"oberlein type and the inclusion functor is even monadic. We also prove that the forgetful functors from $\CocomplSep{\Tth}$ to $\SET$ and to $\CatSep{\V}$ are monadic. At this point we notice that our categorical approach has led us to a well-known result for topological spaces: injective T$_0$-spaces (together with suitable morphisms) are the Eilenberg--Moore algebras for the ``filter on open subsets'' monad on $\TOP_0$, the category of T$_0$-spaces and continuous maps, as well as for the filter monad on $\SET$ (see \cite{Day_Filter,Esc_InjSp} for details). We have now generalised these facts to $\Tth$-categories, but to do so we used (almost) only standard arguments from Category Theory! 

Finally, we wish to highlight a possible application of our work. One of the nice features of domain theory is the strong interaction between topological and order-theoretic ideas. For instance, continuous lattices \cite{Sco_ContLat} can be described purely in order theoretic terms as well as in topological terms: as ordered sets with certain completeness properties, or as injective topological T$_0$-spaces with respect to embeddings. There exist many interesting attempts in the literature to introduce \emph{continuous metric spaces}, or, more general, \emph{continuous $\V$-categories}; all of them are (more or less) based on the order-theoretic approach to continuous lattices (\cite{Wag_PhD,BBR_GenMet,Was_PhD}). We are not aware of any attempt using injectivity properties in a suitable category. The results of our work indicate that, for instance, R.\ Lowen's approach spaces (\cite{Low_ApBook}) can serve as a useful tool for the introduction and study of continuous metric spaces. In fact, as a particular instance of our work we deduce that the injective T$_0$-approach spaces can be described as the cocomplete T$_0$-approach spaces, but also as the Eilenberg--Moore algebras for suitable monads on sets respectively metric spaces. Looking at it from the other end, we obtain a metric equivalent to the filter monad, whose algebras are precisely the injective T$_0$-approach spaces.

\section{The Setting}

\subsection{Topological theories}
Throughout this paper we consider a (strict) \emph{topological theory} as introduced in \cite{Hof_TopTh}. Such a theory $\Tth=\toptheory$ consists of a commutative quantale $\V=\quantale$, a $\SET$-monad $\mT=\monad$ where $T$ and $m$ satisfy \BC\ (that is, $T$ sends pullbacks to weak pullbacks and each naturality square of $m$ is a weak pullback) and a map $\xi:T\V\to\V$ such that
\begin{enumerate}
\item the monoid $\V$ in $\SET$ lifts to a monoid $(\V,\xi)$ in $(\SET^\mT,\times,1)$, that is, $\xi:T\V\to\V$ is a $\mT$-algebra structure on $\V$ and $\otimes:\V\times\V\to\V$ and $k:1\to\V$ are $\mT$-algebra homomorphisms. In orther words, we require the following diagrams to commute.
\begin{align*}
\xymatrix{X\ar[r]^{e_X}\ar[dr]_{1_X} & TX\ar[d]^\xi\\ & X} &&& \hspace{3em}
\xymatrix{TTX\ar[d]_{m_X}\ar[r]^{T\xi} & TX\ar[d]^\xi\\ TX\ar[r]_\xi & X}\\
\xymatrix{T1\ar[d]_{!}\ar[r]^{Tk} & T\V\ar[d]^\xi\\ 1\ar[r]_k & \V} &&&
\xymatrix{T(\V\times\V)\ar[rr]^-{T(\otimes)}\ar[d]_{\langle\xi\cdot T\pi_1,\xi\cdot T\pi_2\rangle} && T\V\ar[d]^\xi\\ \V\times\V\ar[rr]_-{\otimes} && \V}
\end{align*}
\item $\xi_X:=\xi\cdot T(-)$ defines a natural transformation $(\xi_{_X})_{X}:P_{\V}\to P_{\V} T:\SET\to\ORD$.
\end{enumerate}
Here $P_{\V}:\SET\to\ORD$ is the $\V$-powerset functor defined as follows. We put $P_\V(X)=\V^X$ with the pointwise order. For a function $f:X\to Y$, we have a monotone map $\V^f:\V^Y\to\V^X,\,\varphi\mapsto\varphi\cdot f$. It is easy to see that $\V^f$ preserves all infima and all suprema, hence has in particular a left adjoint denoted as $P_\V(f)$. Explicitly, for $\varphi\in\V^X$ we have $P_\V(f)(\varphi)(y)=\bigvee\{\varphi(x)\mid x\in X,\,f(x)=y\}$.
\begin{examples}\label{ExTheories}
\begin{enumerate}
\item The identity theory $\Ith=\itheory$, for each quantale $\V$, where $\mI=\imonad$ denotes the identity monad.
\item\label{ExTop} $\Uth_{\two}=\sptheory$, where $\mU=\umonad$ denotes the ultrafilter monad and $\xi_{\two}$ is essentially the identity map.
\item\label{ExApp} $\Uth_{\Pplus}=(\mU,\Pplus,\xi_{\Pplus})$ where $\Pplus=([0,\infty]^\op,+,0)$ and
\[
\xi_{\Pplus}:U\Pplus\to\Pplus,\;\;\frx\mapsto\inf\{v\in\Pplus\mid[0,v]\in\frx\}.
\]
\item\label{WordTh} The word theory $(\mL,\V,\xi_{_\otimes})$, for each quantale $\V$, where $\mL=\wmonad$ is the word monad and
\begin{align*}
\xi_{_\otimes}:L\V &\to\V.\\
(v_1,\ldots,v_n) &\mapsto v_1\otimes\ldots\otimes v_n\\
() &\mapsto k
\end{align*}
\end{enumerate}
\end{examples}

\subsection{$\V$-relations}
The quantaloid $\Mat{\V}$ \cite{BCRW_VarEnr} has sets as objects, and an arrow $r:X\relto Y$ from $X$ to $Y$ is a \emph{$\V$-relation} $r:X\times Y\to\V$. Composition of $\V$-relations $r:X\relto Y$ and $s:Y\relto Z$ is defined as matrix multiplication
\[s\cdot r(x,z)=\bigvee_{y\in Y}r(x,y)\otimes s(y,z),\]
and the identity arrow $1_X:X\relto X$ is the $\V$-relation which sends all diagonal elements $(x,x)$ to $k$ and all other elements to the bottom element $\bot$ of $\V$. The complete order of $\V$ induces a complete order on $\Mat{\V}(X,Y)=\V^{X\times Y}$: for $\V$-relations $r,r':X\relto Y$ we define  
\[
r\le r'\;:\iff\; \forall x\in X\;\forall y\in Y\;.\;r(x,y)\le r'(x,y).
\]

Any element $u\in\V$ can be interpreted as a $\V$-relation $u:1\relto 1$. Then, given also $v\in\V$, $v\cdot u=v\otimes u$, and $k$ represents the identity arrow. We have an involution $(r:X\relto Y)\mapsto (r^\circ:Y\relto X)$ where $r^\circ(y,x)=r(x,y)$, satisfying
\begin{align*}
1_X^\circ=1_X, && (s\cdot r)^\circ=r^\circ\cdot s^\circ, && {r^\circ}^\circ=r,
\end{align*}
as well as $r^\circ\le s^\circ$ whenever $r\le s$. Furthermore, there is an obvious functor
\[\SET\to\Mat{\V},\;(f:X\to Y)\mapsto (f:X\relto Y)\]
sending a map $f:X\to Y$ to its graph $f:X\relto Y$ defined by
\[
f(x,y)=
\begin{cases}
k & \text{if $f(x)=y$,}\\
\bot & \text{else.}
\end{cases}
\]
Then, in the quantaloid $\Mat{\V}$, we have $f\dashv f^\circ$. If the quantale $\V$ is non-trivial, i.e.\ if $\bot<k$, then the functor above from $\SET$ to $\Mat{\V}$ is faithful and we can identify the function $f:X\to Y$ with the $\V$-relation $f:X\relto Y$. \emph{In the sequel we will always assume $\bot<k$}, and write $f:X\to Y$ for both the function and the $\V$-relation.

Let $t:X\relto Z$ be a $\V$-relation. The composition functions
\begin{align*}
-\cdot t:\Mat{\V}(Z,Y)\to\Mat{\V}(X,Y) &&\text{and}&&
t\cdot-:\Mat{\V}(Y,X)\to\Mat{\V}(Y,Z).
\end{align*}
preserve suprema and therefore have respective right adjoints
\begin{align*}
(-)\homcompright t:\Mat{\V}(X,Y)\to\Mat{\V}(Z,Y) &&\text{and}&&
t\homcompleft(-):\Mat{\V}(Y,Z)\to\Mat{\V}(Y,X).
\end{align*}
Hence, for $\V$-relations $s:Z\relto Y$, $r:X\relto Y$ respectively $s:Y\to X$, $r:Y\relto Z$, we have bijections
\begin{align*}
&\underline{\;s\cdot t\le r\;} &&\text{and}&& \underline{\;t\cdot s \le r\;\;}.\\
&s\le r\homcompright t &&&& s\le t\homcompleft r\\
&\xymatrix{X\ar[rd]|-{\object@{|}}^r\ar[d]|-{\object@{|}}_t & \ \\ Z\ar@{}[ur]|-<<{\le} \ar[r]|-{\object@{|}}_s & Y}
&&&&
\xymatrix{Z & \\ X\ar[u]|-{\object@{|}}^t\ar@{}[ur]|-<<{\le} &
Y\ar[lu]|-{\object@{|}}_r\ar[l]|-{\object@{|}}^s}\\
\end{align*}
We call $r\homcompright t$ the \emph{extension of $r$ along $t$}, and $t\homcompleft r$ the \emph{lifting of $r$ along $t$}.

\subsection{$\Tth$-relations}
The functor $T:\SET\to\SET$ extends to a 2-functor $\Txi:\Mat{\V}\to\Mat{\V}$ as follows: we put $\Txi X=TX$ for each set $X$, and 
\begin{align*}
\Txi r:TX\times TY &\to\V\\
r(\frx,\fry) &\mapsto\bigvee\left\{\xi\cdot Tr(\frw)\;\Bigl\lvert\;\frw\in T(X\times Y), T\pi_1(\frw)=\frx,T\pi_2(\frw)=\fry\right\}
\end{align*}
for each $\V$-relation $r:X\relto Y$. That is, $\Txi r:TX\times TY\to\V$ is the smallest (order-preserving) map $s:TX\times TY\to\V$ such that $\xi\cdot Tr\le s\cdot\can$.
\[
\xymatrix{T(X\times Y)\ar[rr]^\can\ar[dr]_{\xi_{X\times Y}(r)=\xi\cdot Tr} & & 
TX\times TY\ar@{..>}[dl]^{\Txi r}\\ & \V\ar@{}[u]|-{\le}}
\]
As shown in \cite{Hof_TopTh}, we have $\Txi f=Tf$ for each function $f:X\to Y$, $\Txi(r^\circ)=\Txi(r)^\circ$ (and we write $\Txi r^\circ$) for each $\V$-relation $r:X\relto Y$, $m$ becomes a natural transformation $m:\Txi\Txi\to\Txi$ and $e$ an op-lax natural transformation $e:\Id\to\Txi$, i.e.\ $e_Y\circ r\leq \Txi r\circ e_X$ for all $r:X\relto Y$ in $\Mat{\V}$.

A $\V$-relation of the form $\alpha:TX\relto Y$ we call \emph{$\Tth$-relation} from $X$ to $Y$, and write $\alpha:X\krelto Y$. For $\Tth$-relations $\alpha:X\krelto Y$ and $\beta:Y\krelto Z$ we define the \emph{Kleisli convolution} $\beta\kleisli\alpha:X\krelto Z$ as
\[\beta\kleisli\alpha=\beta\cdot \Txi\alpha\cdot m_X^\circ.\]
Kleisli convolution is associative and has the $\Tth$-relation $e_X^\circ:X\krelto X$ as a lax identity: $a\kleisli e_X^\circ=a$ and $e_Y^\circ\kleisli a\ge a$ for any $a:X\krelto Y$. We call $a:X\krelto Y$ \emph{unitary} if $e_Y^\circ\kleisli a= a$, so that $e_X^\circ:X\krelto X$ is the identity on $X$ in the category $\UMat{\Tth}$ of sets and unitary $\Tth$-relations, with the Kleisli convolution as composition. In fact, $\UMat{\Tth}$ is a locally completely 2-category, where the 2-categorical structure is inherited from $\Mat{\V}$. Furthermore, for a $\Tth$-relation $\alpha:X\krelto Y$, the composition function $-\kleisli\alpha$ still has a right adjoint $(-)\homkleisliright\alpha$ but $\alpha\kleisli-$ in general not. Explicitly, given also $\gamma:X\krelto Z$, we pass from
\begin{align*}
\xymatrix{X\ar@{-^{>}}|-{\object@{|}}[r]^\gamma\ar@{-^{>}}|-{\object@{|}}[d]_\alpha & Z\\ Y} &&\text{to}&&
\xymatrix{TX\ar[r]|-{\object@{|}}^\gamma\ar[d]|-{\object@{|}}_{m_X^\circ} & Z\\
TTX\ar[d]|-{\object@{|}}_{\Txi\alpha}\\ TY}
\end{align*}
and define $\gamma\homkleisliright\alpha:=\gamma\homcompright(\Txi\alpha\cdot m_X^\circ)$. One easily verifies the required universal property, which in particular implies that $\gamma\homkleisliright\alpha$ is unitary if $\alpha$ and $\gamma$ are so.

\subsection{$\Tth$-categories}
A \emph{$\Tth$-category} is a pair $(X,a)$ consisting of a set $X$ and a $\Tth$-endorelation $a:X\krelto X$ on $X$ such that
\begin{align*}
e_X^\circ&\le a &&\text{and}& a\kleisli a&\le a.
\end{align*}
Expressed elementwise, these conditions become
\begin{align*}
k&\le a(e_X(x),x) &&\text{and}& \Txi a(\frX,\frx)\otimes a(\frx,x)\le a(m_X(\frX),x)
\end{align*}
for all $\frX\in TTX$, $\frx\in TX$ and $x\in X$. A function $f:X\to Y$ between $\Tth$-categories $(X,a)$ and $(Y,b)$ is a \emph{$\Tth$-functor} if $f\cdot a\le b\cdot Tf$, which in pointwise notation reads as
\[
 a(\frx,x)\le b(Tf(\frx),f(x))
\]
for all $\frx\in TX$, $x\in X$. If we have above even equality, we call $f:X\to Y$ \emph{fully faithful}. The resulting category of $\Tth$-categories and $\Tth$-functors we denote as $\Cat{\Tth}$. The quantale $\V$ becomes a $\Tth$-category $\V=(\V,\hom_\xi)$, where $\hom_\xi:T\V\times\V\to\V,\;(\frv,v)\mapsto\hom(\xi(\frv),v)$ (see \cite{Hof_TopTh}). 
\begin{examples}
\begin{enumerate}
\item For each quantale $\V$, $\Ith_\V$-categories are precisely $\V$-categories and $\Ith_\V$-functors are $\V$-functors. As usual, we write $\V$-category instead of $\Ith_\V$-category, $\V$-functor instead of $\Ith_\V$-functor, and $\Cat{\V}$ instead of $\Cat{\Ith_\V}$.
\item The main result of \cite{Bar_RelAlg} states that $\Cat{\Uth_\two}$ is isomorphic to the category $\TOP$ of topological spaces and continuous maps. In \cite{CH_TopFeat} it is shown that $\Cat{\Uth_{\Pplus}\!}$ is isomorphic to the category $\AP$ of approach spaces and non-expansive maps \cite{Low_ApBook}.
\end{enumerate}
\end{examples}
The category $\SET^\mT$ of $\mT$-algebras and $\mT$-homomorphisms can be embedded into $\Cat{\Tth}$ by regarding the structure map $\alpha:TX\to X$ of an Eilenberg--Moore algebra $(X,\alpha)$ as a $\Tth$-relation $\alpha:X\krelto X$. The $\Tth$-category resulting this way from the free Eilenberg--Moore algebra $(TX,m_X)$ we denote as $|X|$. The forgetful functor $\ForgetSET:\Cat{\Tth}\to\SET,\;(X,a)\mapsto X$ is \emph{topological} (see \cite{AHS}), hence has a left and a right adjoint and $\Cat{\Tth}$ is complete and cocomplete. The free $\Tth$-category on a set $X$ is given by $(X,e_X^\circ)$. In particular, the free $\Tth$-category $(1,e_1^\circ)$ on a one-element set is a generator in $\Cat{\Tth}$ which we denote as $G=(1,e_1^\circ)$. We have a canonical forgetful functor $\ForgetToV:\Cat{\Tth}\to\Cat{\V}$ sending a $\Tth$-category $X=(X,a)$ to its underlying $\V$-category $\ForgetToV\!X=(X,a\cdot e_X)$. Furthermore, $S$ has a left adjoint $\ForgetToVAd:\Cat{\V}\to\Cat{\Tth}$ defined by $\ForgetToVAd\!X=(X,e_X^\circ\cdot\Txi r)$, for each $\V$-category $X=(X,r)$. However, there is yet another interesting functor connecting $\Tth$-categories with $\V$-categories, namely $\MFunctor:\Cat{\Tth}\to\Cat{\V}$ which sends a $\Tth$-category $(X,a)$ to the $\V$-category $(TX,\Txi a\cdot m_X^\circ)$. This functors are used in \cite{CH_Compl} to define the \emph{dual} of a $\Tth$-category $X$:
\[X^\op=\ForgetToVAd(\MFunctor(X)^\op).\]
Clearly, if $\Tth=\Ith_\V$ is the identity theory $\Ith_\V=\itheory$, then $X^\op$ is the usual dual $\V$-category of $X$. It is by no means obvious why the definition above provides us with a ``good'' generalisation of this construction. We take Theorem \ref{CharTMod} as well as the Yoneda lemma for $\Tth$-categories (see Theorem \ref{Yoneda} and Corollary \ref{YonedaLem}) as a reason to believe so.

As studied in \cite{Hof_TopTh}, the tensor product of $\V$ can be transported to $\Cat{\Tth}$ by putting $(X,a)\otimes(Y,b)=(X\times Y,c)$ with
\[
c(\frw,(x,y))=a(\frx,x)\otimes b(\fry,y),
\]
where $\frw\in T(X\times Y)$, $x\in X$, $y\in Y$, $\frx=T\pi_1(\frw)$ and $\fry=T\pi_2(\frw)$. The $\Tth$-category $E=(1,k)$ is a $\otimes$-neutral object, where $1$ is a singleton set and $k:T1\times 1\to\V$ the constant relation with value $k\in\V$. In general, this constructions does not result in a closed structure on $\Cat{\Tth}$; however, the results of \cite{Hof_TopTh} give us the following
\begin{proposition}
For each $\mT$-algebra $X$, $X\otimes-:\Cat{\Tth}\to\Cat{\Tth}$ has a right adjoint $(-)^X:\Cat{\Tth}\to\Cat{\Tth}$. In particular, the structure $\fspstr{-}{-}$ on $\V^{|X|}$ is given by the formula
\[
\fspstr{\frp}{\psi}=
\bigwedge_{\substack{\frq\in T(|X|\times\V^{|X|})\\ \frq\mapsto\frp}}\hom(\xi\cdot T\ev(\frq),\psi(m_X\cdot T\pi_1(\frq))), 
\]
for each $\frp\in T\V^{|X|}$ and $\psi\in\V^{|X|}$. Moreover, for $\frp=e_{\V^{|X|}}(\varphi)$ we have
\[
\fspstr{e_{\V^{|X|}}(\varphi)}{\psi}=\bigwedge_{\frx\in TX}\hom(\varphi(\frx),\psi(\frx)).
\]
\end{proposition}
Furthermore, several maps obtained from the quantale structure on $\V$ become now $\Tth$-functors.
\begin{proposition}\label{SomeFunctors}
The following assertions hold.
\begin{enumerate}
\item Both $k:E\to\V$ and $\otimes:\V\otimes\V\to\V$ are $\Tth$-functors, hence $\V$ is even a monoid in $(\Cat{\Tth},\otimes,E)$.
\item $\xi:|\V|\to\V$ is a $\Tth$-functor.
\item $\bigvee:\V^{|X|}\to\V$ is a $\Tth$-functor, for each set $X$.
\end{enumerate}
\end{proposition}
\begin{proof}
(1) and (2) are easy to prove, (3) is a consequence of \cite[Proposition 6.11]{Hof_TopTh}.
\end{proof}

\subsection{$\Tth$-modules}
Let $X=(X,a)$ and $Y=(Y,b)$ be $\Tth$-categories and $\varphi:X\krelto Y$ be a $\Tth$-relation. We call $\varphi$ a \emph{$\Tth$-module}, and write $\varphi:X\kmodto Y$, if $\varphi\kleisli a\le\varphi$ and $b\kleisli \varphi\le \varphi$. Note that we always have $\varphi\kleisli a\ge\varphi$ and $b\kleisli \varphi\ge \varphi$, so that the $\Tth$-module condition above implies equality. Kleisli convolution is associative, and it follows that $\psi\kleisli\varphi$ is a $\Tth$-module if $\psi:Y\kmodto Z$ and $\varphi:X\kmodto Y$ are so. Furthermore, we have $a:X\kmodto X$ for each $\Tth$-category $X=(X,a)$, and, by definition, $a$ is the identity $\Tth$-module on $X$ for the Kleisli convolution. In other words, $\Tth$-categories and $\Tth$-modules form a category, denoted as $\Mod{\Tth}$, with Kleisli convolution as compositional structure. In fact, $\Mod{\Tth}$ is an ordered category with the structure on $\hom$-sets inherited from $\Mat{\Tth}$. As before, a $\Ith_\V$-module we call simply $\V$-module and write $\varphi:X\modto Y$, and put $\Mod{\V}=\Mod{\Ith_\V}$. Finally, a $\Tth$-relation $\varphi:X\krelto Y$ is unitary precisely if $\varphi$ is a $\Tth$-module $\varphi:(X,e_X^\circ)\kmodto(Y,e_Y^\circ)$ between the corresponding discrete $\Tth$-categories. 
\begin{remark}\label{AdjKleisli}
Since the compositional and the order structure for $\Tth$-modules is as for $\Tth$-relations, for each $\Tth$-module $\varphi:(X,a)\kmodto(Y,b)$ and each $\Tth$-category $Z=(Z,c)$ we have an order-preserving map $-\kleisli\varphi:\Mod{\Tth}(Y,Z)\to\Mod{\Tth}(X,Z)$. One easily verifies that, if $\zeta:(X,a)\kmodto(Z,c)$ is a $\Tth$-modules, then so is $\zeta\homkleisliright\varphi$. Hence, $-\kleisli\varphi$ has a right adjoint $(-)\homkleisliright\varphi$. Furthermore, if $\varphi\dashv\psi$ in $\Mod{\Tth}$, then $-\kleisli\psi\dashv-\kleisli\varphi$ in $\ORD$, and therefore $-\kleisli\varphi=(-)\homkleisliright\psi$.
\end{remark}

Let now $X=(X,a)$ and $Y=(Y,b)$ be $\Tth$-categories and $f:X\to Y$ be a function. We define $\Tth$-relations $f_*:X\krelto Y$ and $f^*:Y\krelto X$ by putting $f_*=b\cdot Tf$ and $f^*=f^\circ\cdot b$ respectively. Hence, for $\frx\in TX$, $\fry\in TY$, $x\in X$ and $y\in Y$, we have $f_*(\frx,y)=b(Tf(\frx),y)$ and $f^*(\fry,x)=b(\fry,f(x))$. Given now $\Tth$-modules $\varphi$ and $\psi$, we obtain 
\begin{align*}
\varphi\kleisli f_*&=\varphi\cdot Tf &&\text{and}&& f^*\kleisli\psi=f^\circ\cdot\psi.
\end{align*}
In particular, $b\kleisli f_*=f_*$ and $f^*\kleisli b=f^*$, as well as $f_*\kleisli f^*=b\cdot Tf\cdot Tf^\circ\cdot\Txi b\cdot m_Y^\circ\le b$. The following lemma can be easily verified.
\begin{lemma}\label{Modules}
The following assertions are equivalent.
\begin{eqcond}
\item $f:X\to Y$ is a $\Tth$-functor.
\item $f_*$ is a $\Tth$-module $f_*:X\kmodto Y$.
\item $f^*$ is a $\Tth$-module $f^*:Y\kmodto X$.
\item $a\le f^*\kleisli f_*$.
\end{eqcond}
\end{lemma}
As a consequence, for each $\Tth$-functor $f:(X,a)\to(Y,b)$ we have an adjunction $f_*\dashv f^*$ in $\Mod{\Tth}$. Moreover, given also a $\Tth$-functor $g:(Y,b)\to(Z,c)$,
\begin{align*}
&g_*\circ f_*=c\cdot Tg\cdot Tf=c\cdot T(g\cdot f)=(g\cdot f)_*\\
\intertext{and}
&f^*\circ g^*=f^\circ\cdot g^\circ\cdot c=(g\cdot f)^\circ\cdot c=(g\cdot f)^*.
\end{align*}
Since also $(1_X)_*=(1_X)^*=a$, we obtain functors
\begin{align*}
 (-)_*:\Cat{\Tth}\to\Mod{\Tth} &&\text{and}&& (-)^*:\Cat{\Tth}^\op\to\Mod{\Tth},
\end{align*}
where $X_*=X=X^*$, for each $\Tth$-category $X$.
\begin{lemma}
A $\Tth$-functor $f:(X,a)\to(Y,b)$ is fully faithful if and only if $1_X^*=f^*\kleisli f_*$.
\end{lemma}
\begin{lemma}\label{AdjModLem}
Consider $\Tth$-modules $\varphi:X\kmodto Y$, $\psi:X\kmodto Z$ and $\alpha:Y\kmodto B$, where $\alpha$ is right adjoint. Then
\[
\alpha\kleisli(\varphi\homkleisliright\psi)=(\alpha\kleisli\varphi)\homkleisliright\psi.
\]
\end{lemma}
\begin{proof}
Let $\beta:B\kmodto Y$ be the left adjoint of $\alpha$. We have to show that the diagram
\[
\xymatrix{\Mod{\Tth}(X,Y)\ar[r]^{(-)\homkleisliright\psi}\ar[d]_{\alpha\kleisli-} &
  \Mod{\Tth}(Z,Y)\ar[d]^{\alpha\kleisli -}\\ \Mod{\Tth}(X,B)\ar[r]_{(-)\homkleisliright\psi} &\Mod{\Tth}(Z,B)}
\]
of right adjoints commutes. But the diagram
\[
\xymatrix{\Mod{\Tth}(X,Y) & \Mod{\Tth}(Z,Y)\ar[l]_{-\kleisli\psi}\\ \Mod{\Tth}(X,B)\ar[u]^{\beta\kleisli -} & \Mod{\Tth}(Z,B)\ar[u]_{\beta\kleisli -}\ar[l]^{-\kleisli\psi}}
\]
of the corresponding left adjoints commutes since Kleisli convolution is associative, and the assertion follows.
\end{proof}
\begin{theorem}[\cite{CH_Compl}]\label{CharTMod}
For $\Tth$-categories $(X,a)$ and $(Y,b)$, and a $\Tth$-relation $\psi:X\krelto Y$, the following assertions are equivalent.
\begin{eqcond}
\item $\psi:(X,a)\kmodto(Y,b)$ is a $\Tth$-module.
\item Both $\psi:|X|\otimes Y\to\V$ and $\psi:X^\op\otimes Y\to\V$ are $\Tth$-functors.
\end{eqcond}
\end{theorem}
Therefore, each $\Tth$-module $\varphi:X\kmodto Y$ defines a $\Tth$-functor
\[
 \mate{\varphi}:Y\to\V^{|X|}
\]
which factors through the embedding $\hat{X}\hrw\V^{|X|}$, where $\hat{X}=\{\psi\in\V^{|X|}\mid \psi:X\kmodto G\}$.
\[
\xymatrix{Y\ar[r]^{\mate{\varphi}}\ar[dr]_ {\mate{\varphi}}& \V^{|X|}\\ & \hat{X}\ar@{_(->}[u]}
\]
In particular, for each $\Tth$-category $X=(X,a)$ we have $a:X\kmodto X$, and therefore obtain the \emph{Yoneda functor}
\[
 \yoneda_X=\mate{a}:X\to\hat{X}.
\]

\begin{theorem}\label{Yoneda}
Let $\psi:X\kmodto Z$ and $\varphi:X\kmodto Y$ be $\Tth$-modules. Then, for all $\frz\in TZ$ and $y\in Y$,
\[
\fspstr{T\mate{\psi}(\frz)}{\mate{\varphi}(y)}=(\varphi\homkleisliright\psi)(\frz,y).
\]
\end{theorem}
\begin{proof}
First note that the diagrams
\begin{align*}
\xymatrix{ & \V\\ TX\times Z\ar[r]_-{1_{TX}\times\mate{\psi}}\ar[ur]^\psi & TX\times\hat{X}\ar[u]_\ev} &&
\xymatrix{TX\times Z\ar[r]^-{1_{TX}\times\mate{\psi}}\ar[d]_{\pi_2} & TX\times\hat{X}\ar[d]^{\pi_2}\\ Z\ar[r]_{\mate{\psi}} & \hat{X}}
\end{align*}
commute, where the right hand side diagram is even a pullback. Then, for $\frz\in TZ$ and $y\in Y$, we have
\begin{align*}
\fspstr{T\mate{\psi}(\frz)}{\mate{\varphi}(y)}
&=\bigwedge_{\substack{\frW\in T(TX\times\hat{X})\\ \frW\mapsto T\mate{\psi}(\frz)}}\hom(\xi\cdot T\ev(\frW),\varphi(m_X\cdot T\pi_1(\frW),y))\\
&=\bigwedge_{\frx\in TX}\bigwedge_{\substack{\frX\in TTX\\ m_X(\frX)=\frx}}\bigwedge_{\substack{\frW\in T(TX\times\hat{X})\\ \frW\mapsto T\mate{\psi}(\frz),\frX}}\hom(\xi\cdot T\ev(\frW),\varphi(\frx,y))\\
&=\bigwedge_{\frx\in TX}\bigwedge_{\substack{\frX\in TTX\\ m_X(\frX)=\frx}}\hom(\bigvee_{\substack{\frW\in T(TX\times Z)\\ \frW\mapsto \frz,\frX}}\xi\cdot T\psi(\frW),\varphi(\frx,y))\\
&=\bigwedge_{\frx\in TX}\hom(\bigvee_{\substack{\frX\in TTX\\ m_X(\frX)=\frx}}\Txi\psi(\frX,\frz),\varphi(\frx,y))\\
&=\bigwedge_{\frx\in TX}\hom(\Txi\psi\cdot m_X^\circ(\frx,\frz),\varphi(\frx,y))\\
&=\varphi\homcompright(\Txi\psi\cdot m_X^\circ)(\frz,y)=(\varphi\homkleisliright\psi)(\frz,y).\qedhere
\end{align*}
\end{proof}
Choosing in particular $\psi=a:X\kmodto X$ and $Y=G$, we obtain the ``usual'' \emph{Yoneda lemma} (see also \cite{CH_Compl}).
\begin{corollary}\label{YonedaLem}
For each $\varphi\in\hat{X}$ and each $\frx\in TX$, $\varphi(\frx)=\fspstr{T\yoneda_X(\frx)}{\varphi}$, that is,$(\yoneda_X)_*:X\kmodto\hat{X}$ is given by the evaluation map $\ev:TX\otimes\hat{X}\to\V$. As a consequence, $\yoneda_X:X\to\hat{X}$ is fully faithful.
\end{corollary}

\section{Cocomplete $\Tth$-categories}

\subsection{$\Cat{\Tth}$ as an ordered category}
We can transport the order-structure on hom-sets from $\Mod{\Tth}$ to $\Cat{\Tth}$ via the functor $(-)^*:\Cat{\Tth}^\op\to\Mod{\Tth}$, that is, we define $f\le g$ whenever $f^*\le g^*$. Clearly, we have $f\le g$ if and only if $g_*\le f_*$. With this definition we turn $\Cat{\Tth}$ into a 2-category, and therefore the (representable) forgetful functor $\ForgetSET:\Cat{\Tth}\to\SET$ factors through $\ForgetSET:\Cat{\Tth}\to\ORD$. As usual, we call $\Tth$-functors $f,g:X\to Y$ \emph{equivalent}\index{functor@$\Tth$-functor!equivalent -s}, and write $f\cong g$, if $f\le g$ and $g\le f$. Hence, $f\cong g$ if and only if $f^*=g^*$, which in turn is equivalent to $f_*=g_*$. We call a $\Tth$-category $X$ \emph{L-separated} (see \cite{HT_LCls} for details) whenever $f\cong g$ implies $f=g$, for all $\Tth$-functors $f,g:Y\to X$ with codomain $X$. The $\Tth$-category $\V=(\V,\hom_\xi)$ is L-separated, and so is each $\Tth$-category of the form $\hat{X}$, for a $\Tth$-category $X$. The full subcategory of $\Cat{\Tth}$ consisiting of all L-separated $\Tth$-categories is denoted by $\CatSep{\Tth}$. A $\Tth$-category $X$ is called \emph{injective} if, for all $\Tth$-functors $f:A\to X$ and fully faithful $\Tth$-functors $i:A\to B$, there exists a $\Tth$-functor $g:B\to X$ such that $g\cdot i\cong f$. Clearly, for a L-separated $\Tth$-category $X$ we have then $g\cdot i= f$.
\begin{lemma}\label{OrdTFun}
The following assertions hold.
\begin{enumerate}
\item Let $f,g:X\to Y$ be $\Tth$-functors between $\Tth$-categories $X=(X,a)$ and $Y=(Y,b)$. Then
\[
f\le g\iff \forall x\in X\;.\;k\le b(e_Y(f(x)),g(x)).
\]
In particular, for $\Tth$-functors $f,g:Y\to\V^{|X|}$ we have
\[
 f\le g\iff \forall y\in Y,\frx\in TX\;.\;f(y)(\frx)\le g(y)(\frx).
\]
\item A $\Tth$-category $X$ is L-separated if and only if the underlying $\V$-category $\ForgetToV X$ is L-separated.
\item With $X$ also $\ForgetToV X$ is injective with respect to fully faithful functors, for each $\Tth$-category $X$.
\end{enumerate}
\end{lemma}
\begin{proof}
(1) can be found in \cite{HT_LCls}, (2) follows immediately from (1), and (3) follows from the facts that $\ForgetToV:\Cat{\Tth}\to\Cat{\V}$ is actually a 2-functor and it's left adjoint $\ForgetToVAd:\Cat{\V}\to\Cat{\Tth}$ sends fully faithful $\V$-functors to fully faithful $\Tth$-functors.
\end{proof}
One of the most important concepts in a 2-category is that of \emph{adjointness}. Here, a $\Tth$-functor $f:X\to Y$ is \emph{left adjoint} if there exists a $\Tth$-functor $g:Y\to X$ such that $1_X\le g\cdot f$ and $1_Y\ge f\cdot g$. Passing to $\Mod{\Tth}$, $f$ is left adjoint to $g$ if and only if $g_*\dashv f_*$, that is, if and only if $f_*=g^*$. Bearing in mind Lemma \ref{Modules}, we have
\begin{proposition}
A $\Tth$-functor $f:X\to Y$ is left adjoint if and only if there exists a function $g:Y\to X$ such that $f_*=g^*$, that is,
\[
 b(Tf(\frx),y)=a(\frx,g(y),
\]
for all $\frx\in TX$ and $y\in Y$.
\end{proposition}

\subsection{Cocomplete $\Tth$-categories}
Let now $X=(X,a)$ be a $\Tth$-category. Given a $\Tth$-functor $h:Y\to X$ and a \emph{weight} $\psi:Y\kmodto Z$ in $\Mod{\Tth}$,
\[
\xymatrix{Y\ar@{-^{>}}|-{\object@{o}}[r]^{h_*}\ar@{-^{>}}|-{\object@{o}}[d]_\psi & X\\ Z\ar@{.^{>}}|-{\object@{o}}[ur]_{h_*\homkleisliright\psi}} 
\]
we call a $\Tth$-functor $g:Z\to X$ \emph{a $\psi$-weighted colimit of $h$}, and  write $g\cong\colim(\psi,h)$, if $g$ represents $h_*\homkleisliright\psi$, i.e.\ if $h_*\homkleisliright\psi=g_*$. Clearly, if such $g$ exists, it is unique up to equivalence and therefore we call  $g$  ``the'' $\psi$-weighted colimit of $h$. We say that a $\Tth$-functor $f:X\to Y$ \emph{preserves the $\psi$-weighted colimit of $h$} if $f\cdot\colim(\psi,h)\cong\colim(\psi,f\cdot h)$, that is, if $(f\cdot g)_*=(f\cdot h)_*\homkleisliright\psi$. A $\Tth$-functor $f:X\to Y$ is \emph{cocontinuous} if $f$ preserves all weighted colimits which exist in $X$, and a $\Tth$-category $X$ is \emph{cocomplete} if each ``weighted diagram'' has a colimit in $X$. A straightforward calculation shows that we only need to consider $f=1_X$.
\begin{lemma}
Let $f:Y\to X$ be a $\Tth$-functor and $\psi:Y\kmodto Z$ be a $\Tth$-module. Then $\colim(\psi,f)\cong\colim(\psi\circ f^*,1_X)$. In particular, $X$ is cocomplete if and only if $1_X^*\homkleisliright\psi$ is representable by some $\Tth$-functor $g:Z\to X$, for each $\Tth$-module $\psi:X\kmodto Z$. Furthermore, a $\Tth$-functor $f:X\to Y$ is cocontinuous if and only if $f$ preserves all $\psi$-weighted colimits of $1_X$.
\end{lemma}
\begin{remark}
When studying $\V$-categories, one can go even one step further and show that cocompleteness reduces to the case $Z=G$. More precise, a $\V$-category $X$ is cocomplete if and only if $(1_X)^*\homcompright\psi$ is representable by some $\V$-functor, for each $\V$-module $\psi:X\modto G$. However, for a general theory $\Tth$ I am not able to prove this.
\end{remark}
We let $\Cocompl{\Tth}$ denote the 2-category of all cocomplete $\Tth$-categories and left adjoint $\Tth$-functors between them. Correspondingly, $\CocomplSep{\Tth}$ denotes the full subcategory of $\Cocompl{\Tth}$ consisting of all L-separated cocomplete $\Tth$-categories.
\begin{proposition}\label{SomeColims}
The following assertions hold.
\begin{enumerate}
\item Each $\mate{\psi}\in\hat{X}$ is a colimit of represantables. More precisely, we have $\yoneda_*\homkleisliright\psi={\mate{\psi}}_*$.
\[
\xymatrix{X\ar@{-^{>}}|-{\object@{o}}[r]^{\yoneda_*}\ar@{-^{>}}|-{\object@{o}}[d]_\psi & \hat{X}\\ G\ar@{.^{>}}|-{\object@{o}}[ur]_{\yoneda_*\homkleisliright\psi}} 
\]
\item A left adjoint $\Tth$-functor $f:X\to Y$ between $\Tth$-categories is cocontinuous.
\end{enumerate}
\end{proposition}
\begin{proof}
(1)\hspace{1ex} Let $\fra\in T1$ and $h\in\hat{X}$. Then, by Theorem \ref{Yoneda},
\[
(\yoneda_*\homkleisliright\psi)(\fra,h)=\fspstr{T\mate{\psi}(\fra)}{h}={\mate{\psi}}_*(\fra,h).
\]
(2)\hspace{1ex} Let $h:A\to X$ be in $\Cat{\Tth}$, $\psi:A\kmodto B$ in $\Mod{\Tth}$, and $g\cong\colim(\psi,h)$. Then, since $f_*$ is a right adjoint $\Tth$-module, from Lemma \ref{AdjModLem} we deduce
\[
(f\cdot h)_*\homkleisliright\psi
=f_*\kleisli (h_*\homkleisliright\psi)
=f_*\kleisli g_*=(f\cdot g)_*.\qedhere
\]
\end{proof}
\begin{theorem}\label{CharCocompl}
Let $X=(X,a)$ be a $\Tth$-category. The following assertions are equivalent.
\begin{eqcond}
\item $X$ is injective.
\item $\yoneda_X:X\to\hat{X}$ has a left inverse, i.e.\ there exists a $\Tth$-functor $\Sup_X:\hat{X}\to X$ such that $\Sup_X\cdot\yoneda_X\cong 1_X$.
\item $\yoneda_X:X\to\hat{X}$ has a left adjoint $\Sup_X:\hat{X}\to X$.
\item $X$ is cocomplete.
\end{eqcond}
\end{theorem}
\begin{proof}
(i)$\Rw$(ii)\hspace{1ex} Follows immediately from the fact that $\yoneda_X:X\to\hat{X}$ is fully faithful (see Corollary \ref{YonedaLem}).\\
(ii)$\Rw$(iii)\hspace{1ex} Since $\Sup_X\cdot\yoneda_X\cong 1_X$ by hypothesis, it is enough to show $1_{\hat{X}}\le\yoneda_X\cdot \Sup_X$. Let $\psi\in\hat{X}$ and $\frx\in TX$. Then, by Corollary \ref{YonedaLem} and Lemma \ref{OrdTFun}, we have
\[
 \psi(\frx)=\fspstr{T\yoneda_X(\frx)}{\psi}\le a(T(\Sup_X\cdot\yoneda)(\frx),\Sup_X(\psi))=a(\frx,\Sup_X(\psi))
=\fspstr{T\yoneda_X(\frx)}{\yoneda_X\cdot \Sup_X(\psi)}=\yoneda_X\cdot \Sup_X(\psi)(\frx).
\]
(iii)$\Rw$(iv)\hspace{1ex} Assume $\Sup_X\dashv\yoneda_X$ and let $\psi:X\kmodto Y$ in $\Mod{\Tth}$. By Theorem \ref{Yoneda}, for all $\fry\in TY$ and $x\in X$ we have
\begin{multline*}
1_X^*\homkleisliright\psi(\fry,x)
=\fspstr{T\mate{\psi}(\fry)}{\yoneda_X(x)}
=\yoneda_X^\circ\cdot{\mate{\psi}}_*(\fry,x)
=\yoneda_X^*\kleisli{\mate{\psi}}_*(\fry,x)\\
=(\Sup_X)_*\kleisli{\mate{\psi}}_*(\fry,x)
=(\Sup_X\cdot\mate{\psi})_*(\fry,x),
\end{multline*}
hence $\Sup_X\cdot\mate{\psi}\cong\colim(\psi,1_X)$.\\
(iv)$\Rw$(i)\hspace{1ex}  Let $i:A\to B$ be a fully faithful $\Tth$-functor. Let $f:A\to X$ be a $\Tth$-functor. Hence, by cocompleteness of $X$, $f_*\homkleisliright i_*=g_*$ for some $\Tth$-functor $g:B\to X$. Hence $(g\cdot i)_*=g_*\kleisli i_*\le f_*$. On the other hand, from $f_*=f_*\kleisli i^*\kleisli i_*$ we deduce $f_*\kleisli i^*\le f_*\homkleisliright i_*=g_*$, hence $f_*\le g_*\kleisli i_*$.
\end{proof}
\begin{remarks}
As it happens often, the proof of the theorem above gives us some further information. Firstly, any left inverse $S:\hat{X}\to X$ to the Yoneda embedding $\yoneda_X:X\to\hat{X}$ is actually left adjoint to $\yoneda_X$. I learned this useful fact in the context of quantaloid-enriched categories from Isar Stubbe. Secondly, the $\psi$-weighted colimit of $1_X:X\to X$ in a cocomplete $\Tth$-category $X$ can be calculated as $\Sup_X\cdot\mate{\psi}$. Finally, if $X$ is injective, then any $\Tth$-functor $f:A\to X$ has not only an extension along a fully faithful $\Tth$-functor $i:A\to B$, but even a smallest one with respect to the order on hom-sets in $\Cat{\Tth}$.
\end{remarks}
Let $f:X\to Y$ be a function. We define $f^{-1}:\V^{|Y|}\to\V^{|X|}$ to be the mate of the composite
\[
 |X|\otimes\V^{|Y|}\xrightarrow{\hspace{1em}|f|\otimes 1_{\V^{|Y|}}\hspace{1em}} |Y|\otimes\V^{|Y|} \xrightarrow{\hspace{1em}\ev\hspace{1em}}\V
\]
of $\Tth$-functors. Explicitly, for any $\psi\in\V^{|Y|}$ and $\frx\in TX$, $f^{-1}(\psi)(\frx)=\psi(Tf(\frx))$. Hence, if $f$ is a $\Tth$-functor and $\psi\in\hat{Y}$, then $f^{-1}(\psi)=\psi\kleisli f_*\in\hat{X}$, so hat $f^{-1}$ restricts to a $\Tth$-functor
\[
 f^{-1}:\hat{Y}\to\hat{X}.
\]
\begin{theorem}\label{HatVCocompl}
For each $\Tth$-category $X$, $\hat{X}$ is cocomplete where $\Sup_{\hat{X}}=\yoneda_X^{-1}$.
\end{theorem}
\begin{proof}
According to Theorem \ref{CharCocompl}, we have to show $\yoneda_X^{-1}\cdot\yoneda_{\hat{X}}=1_{\hat{X}}$. To do so, let $\psi\in\hat{X}$ and $\frx\in TX$. Then, by the Yoneda Lemma (Corollary \ref{YonedaLem}), we have
\[
\yoneda_X^{-1}(\yoneda_{\hat{X}}(\psi))(\frx)
=\yoneda_{\hat{X}}(\psi)(T\yoneda_X(\frx))
=\fspstr{T\yoneda_X(\frx)}{\psi}=\psi(\frx),
\]
and the assertion follows.
\end{proof}
Note that the Theorem above applies in particular to the discrete $\Tth$-category $X=(X,e_X^\circ)$, hence $\V^{|X|}$ is cocomplete for each set $X$. Clearly, if $T1=1$, then $\V^{|1|}\cong\V$ and therefore the $\Tth$-category $\V$ is cocomplete and hence injective in $\Cat{\Tth}$. A different proof of this property of $\V$ can be found in \cite[Lemma 3.18]{HT_LCls}. Note that also in the proof of \cite{HT_LCls} the condition $T1=1$ is crucial.

\subsection{Kan extension}
From Theorem \ref{CharCocompl} we know that each $\Tth$-functor $f:X\to Y$ into a cocomplete $\Tth$-category $Y$ has a smallest extension along $\yoneda_X:X\to\hat{X}$. We will see now that this extension is particularly nice (compare with \cite[Theorem 5.35]{Kel_EnrCat}).
\begin{theorem}\label{Kan}
Composition with $\yoneda_X:X\to\hat{X}$ defines an equivalence
\[
 \Cocompl{\Tth}(\hat{X},Y)\to\Cat{\Tth}(X,Y)
\]
of ordered sets, for each cocomplete $\Tth$-category $Y$. That is, for each $\Tth$-functor $f:X\to Y$ into a cocomplete $\Tth$-category $Y$, there exists a (up to equivalence) unique left adjoint $\Tth$-functor $f_L:\hat{X}\to Y$ such that $f_L\cdot\yoneda_X\cong f$; and, if $f\le f'$, then $f_L\le f'_L$. Moreover, the right adjoint to $f_L$ is given by $\mate{f_*}$.
\[
\xymatrix{X\ar[r]^{\yoneda_X}\ar[dr]_f^\cong &\hat{X}\ar[d]_{f_L}^\dashv\\ & Y\ar@/_1pc/[u]_{\mate{f_*}}}
\]
\end{theorem}
\begin{proof}
Let $f_L:\hat{X}\to Y$ be the extension of $f$ where $(f_L)_*=f_*\homkleisliright(\yoneda_X)_*$. Then, by Theorem \ref{Yoneda}, for any $\frp\in T\hat{X}$ and $y\in Y$, we have
\[
 (f_L)_*(\frp,y)=f_*\homkleisliright(\yoneda_X)_*(\frp,y)
=\fspstr{\frp}{\mate{f_*}(y)}
={\mate{f_*}}^*(\frp,y),
\]
hence $f_L\dashv\mate{f_*}$. Unicity of $f_L$ follows from Proposition \ref{SomeColims}. Assume now $f\le f'$. Then $f'_*\le f_*$ and therefore $(f'_L)_*\kleisli(\yoneda_X)_*\le f'_*\le f_*$. Hence $(f'_L)_*\le(f_L)_*$, that is, $f_L\le f'_L$.
\end{proof}
The theorem above tells us that both inclusion functors $\CocomplSep{\Tth}\hrw\CatSep{\Tth}$ and $\CocomplSep{\Tth}\hrw\Cat{\Tth}$ have a left adjoint defined by $X\mapsto\hat{X}$ which, moreover, is a 2-functor. In particular, if $f:X\to Y$ is a $\Tth$-functor, then $\yoneda_Y\cdot f:X\to\hat{Y}$ has a left adjoint extension $\hat{f}:\hat{X}\to\hat{Y}$ along $\yoneda_X:X\to\hat{X}$.
\[
 \xymatrix{X\ar[r]^{\yoneda_X}\ar[d]_f
	& \hat{X}\ar[d]^{\hat{f}}\\ Y\ar[r]_{\yoneda_Y} & \hat{Y}}
\]
Furthermore, by Theorem \ref{Kan}, the right adjoint of $\hat{f}$ is given by $\mate{(\yoneda_Y\cdot f)_*}:\hat{Y}\to\hat{X}$. Explicitly, for each $\psi\in\hat{Y}$ and each $\frx\in TX$ we have
\[
 \mate{(\yoneda_Y\cdot f)_*}(\psi)(\frx)
=(\yoneda_Y)_*\kleisli f_*(\frx,\psi)
=(\yoneda_Y)_*\cdot Tf(\frx,\psi)
=(\yoneda_Y)_*(Tf(\frx),\psi)
=\psi(Tf(\frx)),
\]
that is, $f^{-1}=\mate{(\yoneda_Y\cdot f)_*}$ and $\hat{f}\dashv f^{-1}$. Passing to the underlying ordered sets, $f^{-1}:\hat{Y}\to\hat{X}$ corresponds to $-\kleisli f_*$, therefore the underlying (order-preserving) map of $\hat{f}$ is given by $-\kleisli f^*$ (see Remark \ref{AdjKleisli}). Hence, for $\psi\in\hat{X}$ and $\fry\in TY$ we have
\[
\psi\kleisli f^*=\psi\kleisli(f^\circ\cdot b)=\psi\cdot Tf^\circ\cdot\Txi b\cdot m_Y^\circ=\psi\cdot Tf^\circ\cdot s
\]
and
\[
\psi\kleisli f^*(\fry)=\bigvee_{\frx\in TX}\psi(\frx)\otimes s(\fry,Tf(\frx)),
\]
where $b$ denotes the structure on $Y$ and $s=\Txi b\cdot m_Y$. 

Consider now the discrete $\Tth$-category $X_D=(X,e_X^\circ)$. Then, for any $\Tth$-category $X$, the identity map $j_X:X_D\to X,\;x\mapsto x$ is a $\Tth$-functor, and we obtain a left adjoint $\Tth$-functor $\widehat{j_X}:\widehat{X_D}=\V^{|X|}\to\hat{X}$. In the sequel we find it convenient to write $R_X$ instead. One easily verifies that its right adjoint $j_X^{-1}:\hat{X}\to\V^{|X|}$ is given by the inclusion map $i_X:\hat{X}\hrw\V^{|X|}$. 
\begin{corollary}
For each $\Tth$-category $X=(X,a)$, the inclusion functor $i_X:\hat{X}\to\V^{|X|}$ has a left adjoint given by
\[
R_X:\V^{|X|}\to\hat{X},\,\psi\mapsto\left(\frx\mapsto\bigvee_{\fry\in TX}\psi(\fry)\otimes r(\frx,\fry)\right),
\]
where $r=\Txi a\cdot m_X^\circ$.
\end{corollary}
\begin{corollary}
For each function $f:X\to Y$, the left adjoint to $f^{-1}:\V^{|Y|}\to\V^{|X|}$ is given by
\[
\V^{|X|}\to\V^{|Y|},\,\psi\mapsto \left(\fry\mapsto\bigvee_{\frx:Tf(\frx)=\fry}\psi(\frx)\right).
\]
\end{corollary}
For a $\Tth$-functor $f:X\to Y$, let us write temporarily $f_D:(X,e_X^\circ)\to(Y,e_Y^\circ)$ for the same map between the discrete $\Tth$-categories. Since obviously $j_Y\cdot f_D=f\cdot j_X$, we have a commutative diagram
\[
\xymatrix{\V^{|X|}\ar[r]^{\widehat{f_D}}\ar[d]_{R_X} & \V^{|Y|}\ar[d]^{R_Y}\\
  \hat{X}\ar[r]_{\hat{f}} & \hat{Y}}
\]
of $\Tth$-functors. Furthermore, we have $\widehat{f}\cdot f^{-1}=1_{\hat{X}}$ provided that $f$ is L-dense, i.e.\ $f_*\kleisli f^*=1_X^*$. Satisfying \BC, the functor $T:\SET\to\SET$ sends surjections to surjections, and therefore each surjective $\Tth$-functor $f$ is L-dense.

\subsection{Cocomplete $\Tth$-categories as Eilenberg--Moore algebras}\label{CocomplAlg}

\begin{proposition}\label{CharLeftAdFun}
Let $f:X\to Y$ be a $\Tth$-functor between cocomplete $\Tth$-categories. Then the following assertions are equivalent.
\begin{eqcond}
\item $f$ is left adjoint.
\item $f$ is cocontinuous, that is, $f$ preserves all weighted colimits.
\item We have $f\cdot \Sup_X\cong \Sup_Y\cdot\hat{f}$, where $\Sup_X\dashv\yoneda_X$ and $\Sup_Y\dashv\yoneda_Y$.
\[
 \xymatrix{\hat{X}\ar[r]^{\hat{f}}\ar[d]_{\Sup_X}\ar@{}[dr]|{\cong}
	& \hat{Y}\ar[d]^{\Sup_Y}\\ X\ar[r]_f & Y}
\]
\end{eqcond}
\end{proposition}
\begin{proof}
The implication (i)$\Rw$(ii) we proved already in Proposition \ref{SomeColims}. To see that (ii)$\Rw$(iii), recall that $\Sup_X\cong\colim((\yoneda_X)_*,1_X)$ and therefore $f\cdot \Sup_X\cong\colim((\yoneda_X)_*,f)$. With the help of Lemma \ref{AdjModLem}, we get
\[
(f\cdot \Sup_X)_*
=f_*\homkleisliright(\yoneda_X)_*
=(\yoneda_Y^*\kleisli(\yoneda_Y\cdot f)_*)\homkleisliright(\yoneda_X)_*
=\yoneda_Y^*\kleisli((\yoneda_Y\cdot f)_*\homkleisliright(\yoneda_X)_*)
=\yoneda_Y^*\kleisli{\hat{f}}_*=(\Sup_Y\cdot\hat{f})_*.
\]
Finally, to obtain (iii)$\Rw$(i), we show that $f\dashv \Sup_X\cdot f^{-1}\cdot\yoneda_Y$. In fact,
\[
(\Sup_X\cdot f^{-1}\cdot\yoneda_Y)^*
=\yoneda_Y^*\kleisli {f^{-1}}^*\kleisli \Sup_X^*
={\Sup_Y}_*\kleisli {\hat{f}}_*\kleisli \Sup_X^*
=f_*\kleisli {\Sup_X}_*\kleisli \Sup_X^*
=f_*\kleisli \yoneda_X^*\kleisli \Sup_X^*=f_*.\qedhere
\]
\end{proof}
\begin{example}
Recall from Subsection \ref{Kan} that, for each $\Tth$-functor $f:X\to Y$, we have an adjunction $\hat{f}\dashv f^{-1}$ in $\Cat{\Tth}$. The underlying (order-preserving) maps of $\hat{f}$ and $f^{-1}$ are given by $-\kleisli f^*$ and $-\kleisli f_*$ respectively. Furthermore, we have $\hat{\hat{f}}\dashv \widehat{f^{-1}}$. Since $\yoneda_Y\cdot f=\hat{f}\cdot\yoneda_X$, we obtain $\widehat{\yoneda_Y}\cdot\hat{f}=\hat{\hat{f}}\cdot\widehat{\yoneda_X}$ and therefore $\yoneda_X^{-1}\cdot\widehat{f^{-1}}=f^{-1}\cdot\yoneda_Y^{-1}$. Hence, by Theorem \ref{HatVCocompl} and Proposition \ref{CharLeftAdFun}, $f^{-1}$ has a right adjoint $f_\bullet:\hat{X}\to\hat{Y}$ in $\Cat{\Tth}$. The underlying order-preserving map of $f_\bullet$ we identified in Remark \ref{AdjKleisli} as $(-)\homkleisliright f_*$.
\end{example}

The pair of adjoint functors $\CocomplSep{\Tth}\hrw\CatSep{\Tth}$ and $\widehat{(-)}:\CatSep{\Tth}\hrw\CocomplSep{\Tth}$ induces monad on $\CatSep{\Tth}$, denoted as $\mInj=(\widehat{(-)},\yoneda,\mu)$. By Theorem \ref{Kan}, we have that $f\le g$ implies $\hat{f}\le\hat{g}$, so that $\widehat{(-)}$ is a 2-functor. Furthermore, since obviously $\yoneda_{\hat{X}}\cdot\yoneda_X=\yoneda_{\hat{X}}\cdot\yoneda_X$, we have $(\yoneda_{\hat{X}})_*\le(\widehat{\yoneda_X})_*$, that is, $\widehat{\yoneda_X}\le\yoneda_{\hat{X}}$. In general, a monad $\mS=(S,d,l)$ on a locally thin 2-category $\catfont{X}$ is of \emph{Kock-Z\"oberlein type} (see \cite{Koc_MonAd}) if $S$ is a 2-functor and $Sd_X\le d_{SX}$, for all $X\in\catfont{X}$. In fact, in \cite{Koc_MonAd} it is shown that
\begin{theorem}\label{Kock}
Let $\mS=(S,d,l)$ be a monad on a locally thin 2-category $\catfont{X}$ where $S$ is a 2-functor. Then the following assertions are equivalent.
\begin{eqcond}
\item $Sd_X\le d_{SX}$ for all $X\in\catfont{X}$.
\item $Sd_X\dashv l_X$ for all $X\in\catfont{X}$.
\item $l_X\dashv d_{SX}$ for all $X\in\catfont{X}$.
\item For all $X\in\catfont{X}$, a $\catfont{X}$-morphism $h:SX\to X$ is the structure morphism of a $\mS$-algebra if and only if $h\dashv d_X$ with $h\cdot d_X=1_X$.
\end{eqcond}
\end{theorem}
The considerations above tell us that the monad $\mInj=(\widehat{(-)},\yoneda,\mu)$ on $\CatSep{\Tth}$ is of Kock-Z\"oberlein type. Furthermore, by Theorem \ref{CharCocompl} and Proposition \ref{CharLeftAdFun} we have
\begin{theorem}
$(\CatSep{\Tth})^\mInj\cong\CocomplSep{\Tth}$. Hence, in particular, $\CocomplSep{\Tth}$ is complete.
\end{theorem}
Theorem \ref{Kock} also helps us to compute the multiplication $\mu$ of $\mInj$: for any (L-separated) $\Tth$-category $X$ we have $\widehat{\yoneda_X}\dashv\mu_X$ and $\widehat{\yoneda_X}\dashv\yoneda_X^{-1}$, hence $\mu_X=\yoneda_X^{-1}$.

\subsection{Example: topological spaces}
We consider now $\Tth=\Uth_\two=\sptheory$. Hence $\Cat{\Tth}=\TOP$ is the category of topological spaces and continuous maps, and $\CatSep{\Tth}=\TOP_0$ its full subcategory of T$_0$-spaces (see also \cite{CH_Compl,HT_LCls}). Then $M(X)=(UX,\le)$ is the ordered set with
\[
 \frx\le\fry \iff \{\overline{A}\mid A\in\frx\}\subseteq\fry,
\]
and the topology on $|X|$ is given by the Zariski-closure defined by
\[
 \frx\in\cl\calA :\iff \bigcap\calA\subseteq \frx\iff \frx\subseteq\bigcup\calA.
\]
In \cite{HT_LCls} we observed already that the down-closure as well as the up-closure of a Zariski-closed set is again Zariski-closed. A presheaf $\psi\in\hat{X}$ can be identified with the Zariski-closed and down-closed subset $\calA=\psi^{-1}(1)\subseteq UX$, and we consider 
\[
\hat{X}=\{\calA\subseteq UX\mid \text{ $\calA$ is Zariski-closed and down-closed}\}.
\]
The topology on $\hat{X}$ is the \emph{compact-open topology}, which has as basic open sets
\begin{align*}
B(\calB,\{0\})=\{\calA\in\hat{X}\mid \calA\cap\calB=\varnothing\},&&\text{$\calB\subseteq UX$ Zariski-closed.} 
\end{align*}
The Yoneda map $\yoneda_X:X\to\hat{X}$ is given by $\yoneda_X(x)=\{\frx\in UX\mid \frx\rightarrow x\}$. For $\frx\in UX$, $U\yoneda_X(\frx)$ is the ultrafilter generated by the sets
\begin{align*}
\{\{\fra\mid \fra\rightarrow x\}\mid x\in A\} &&(A\in\frx),
\end{align*}
and the Yoneda lemma (Corollary \ref{YonedaLem}) states that it converges to $\calA\in\hat{X}$ precisely if $\frx\in\calA$.

We have maps
\begin{align*}
\Phi_X:P(UX)\to FX,\,\calA\mapsto\bigcap\calA &&\text{and}&&
\Pi_X:FX\to P(UX),\,\frf\mapsto\{\frx\in UX\mid \frf\subseteq\frx\}.
\end{align*}
where $P(UX)$ denotes the powerset of $UX$ and $FX$ the set of all (possibly improper) filters on $X$. Clearly, we have $\frf=\Phi_X(\Pi_X(\frf))$ and $\calA\subseteq\Pi_X(\Phi_X(\calA))$ for $\frf\in FX$ and $\calA\in P(UX)$. Furthermore, $\calA=\Pi_X(\Phi_X(\calA))$ if and only if $\calA$ is Zariski-closed. We let $F_0X$ denote the set of all filters on the lattice $\tau$ of open sets of a topological space $X$, and $F_1X$ the set of all filters on the lattice $\sigma$ of closed sets of $X$. For each filter $\frf$ on $X$ we can consider $\frf\cap\tau\in F_0X$ and $\frf\cap\sigma\in F_1X$, and $\frf$ is determined by this restriction precisely if $\frf$ has a basis of open respectively closed sets. In \cite{HT_LCls} we showed that $\frf=\bigcap\calA$ has a basis of open sets if and only if $\calA$ is down-closed, and $\frf$ has a 
basis of closed sets if and only if $\calA$ is up-closed. Hence
\begin{align*}
\hat{X}\cong F_0X &&\text{and}&& \{\calA\subseteq UX\mid\text{ $\calA$ is Zariski-closed and up-closed}\}\cong F_1 X,
\end{align*}
and the first homeomorphism we also denote as $\Phi_X:\hat{X}\to F_0 X,\,\calA\mapsto (\bigcap\calA)\cap\tau$. Let $B(\calB,\{0\})$ be a basic open set of the topology of $\hat{X}$. Since $B(\calB,\{0\})=B(\upc\calB,\{0\})$, we can assume that $\calB$ is up-closed. Hence, under the bijections above, $F_0(X)$ has
\begin{align*}
\{\frf\in F_0(X)\mid \exists A\in\frf,B\in\frg\,.\,A\cap B=\varnothing\}&&(\frg\in F_1(X))
\end{align*}
as basic open sets. Clearly, it is enough to consider $\frg=\doo{B}$ the principal filter induced by a closed set $B$, so that all sets
\begin{align*}
\{\frf\in F_0(X)\mid \exists A\in\frf\,.\,A\cap B=\varnothing\}=\{\frf\in F_0(X)\mid X\setminus B\in\frf\}&& \text{($B\subseteq X$ closed)}
\end{align*}
form a basis for the topology on $F_0(X)$. We have shown that our presheaf space $\hat{X}$ is homeomorphic to the filter space $F_0(X)$ considered in \cite{Esc_InjSp}. Furthermore, for a continuous map $f:X\to Y$, $f^{-1}:\hat{Y}\to\hat{X}$ corresponds to $f^{-1}:F_0 Y\to F_0 X,\,\frg\mapsto\{f^{-1}(B)\mid B\in\frg\}$ in the sense that the diagram
\[
\xymatrix{\hat{Y}\ar[r]^{\Phi_Y}\ar[d]_{f^{-1}} & F_0 Y\ar[d]^{f^{-1}}\\ \hat{X}\ar[r]_{\Phi_X} & F_0 X}
\]
commutes. Hence, since $\hat{f}\dashv f^{-1}$ as well as $F_0 f\dashv f^{-1}$, $\Phi=(\Phi_X)_X$ is a natural isomorphism from $\widehat{(-)}:\TOP_0\to\TOP_0$ to $F_0:\TOP_0\to\TOP_0$. Since $\Phi_X(\yoneda(x))=\{U\in\tau\mid x\in U\}$ is the neighborhood filter of $x\in X$, the monad $\mInj=(\widehat{(-)},\yoneda,\yoneda^{-1})$ is isomorphic to the filter monad on $\TOP_0$ considered in \cite{Esc_InjSp}.

\subsection{Cocomplete $\Tth$-categories are algebras over $\SET$ and $\CatSep{\V}$}\label{MonOverSet}
We are now aiming to prove that the forgetful functor
\[
G:\CocomplSep{\Tth}\to\SET
\]
is monadic. Clearly, $G$ has a left adjoint given by the composite
\[
 \SET\xrightarrow{\hspace{1em}\text{discrete}\hspace{1em}}\CatSep{\Tth}
\xrightarrow{\hspace{1em}\widehat{(-)}\hspace{1em}}\CocomplSep{\Tth}.
\]
Furthermore, we have the following elementary facts.
\begin{lemma}
Let $f:X\to Y$ and $g:Y\to X$ be $\Tth$-functors with $f\dashv g$ where $X$, $Y$ are L-separated.
\begin{enumerate}
\item The following assertions are equivalent.
\begin{enumerate}
\renewcommand{\theenumii}{\roman{enumii}}
\item $f$ is an epimorphism in $\CatSep{\Tth}$.
\item $f\cdot g=1_Y$.
\item $f$ is a split epimorphism in $\CatSep{\Tth}$.
\end{enumerate}
\item The following assertions are equivalent.
\begin{enumerate}
\renewcommand{\theenumii}{\roman{enumii}}
\item $f$ is a monomorphism in $\CatSep{\Tth}$.
\item $g\cdot f=1_X$.
\item $f$ is a split monomorphism in $\CatSep{\Tth}$.
\end{enumerate}
\end{enumerate}
\end{lemma}
\begin{proof}
From $f\dashv g$ we obtain $f\cdot g\cdot f=f$. If $f$ is an epimorphism in $\CatSep{\Tth}$, then $f\cdot g=1_Y$; if $f$ is a monomorphism in $\CatSep{\Tth}$, then $g\cdot f=1_X$.
\end{proof}
\begin{corollary}
$G$ reflects isomorphisms.
\end{corollary}
\begin{proof}
If $f:X\to Y$ in $\CocomplSep{\Tth}$ is bijective, then $f$ is an isomorphism in $\CatSep{\Tth}$ and therefore also in $\CocomplSep{\Tth}$.
\end{proof}
In order to conclude that $G$ is monadic, it is left to show that $\CocomplSep{\Tth}$ has and $G$ preserves coequaliser of $G$-equivalence relations (see, for instance, \cite[Corollary 2.7]{MS_Monads}). Hence, let $\pi_1,\pi_2:R\rightrightarrows X$ in $\CocomplSep{\Tth}$ be an equivalence relation in $\SET$, where $\pi_1$ and $\pi_2$ are the projection maps. Let $q:X\to Q$ be its coequaliser in $\Cat{\Tth}$. The following fact will be crucial in the sequel:
\begin{equation}\label{SplitFork}
\xymatrix{\hat{R}\ar@<0.5ex>[r]^{\widehat{\pi_1}}\ar@<-0.5ex>[r]_{\widehat{\pi_2}} &\hat{X}\ar[r]^{\hat{q}} &\hat{Q}}
\text{\hspace{1em} is a split fork in $\CatSep{\Tth}$.}
\end{equation}
The splitting here is given by $q^{-1}:\hat{Q}\to\hat{X}$ and $\pi_1^{-1}:\hat{X}\to\hat{R}$. First note that, since both $\pi_1$ and $q$ are surjective, we have $\hat{q}\cdot q^{-1}=1$ and $\widehat{\pi_1}\cdot\pi_1^{-1}=1$. Hence, in order to obtain \eqref{SplitFork}, we need to show
\[
 q^{-1}\cdot \hat{q}=\widehat{\pi_2}\cdot\pi_1^{-1}.
\]
Note that we have $\hat{q}=\hat{q}\cdot\widehat{\pi_1}\cdot\pi_1^{-1}= \hat{q}\cdot\widehat{\pi_2}\cdot\pi_1^{-1}$, and therefore
\[
q^{-1}\cdot \hat{q}= q^{-1}\cdot\hat{q}\cdot\widehat{\pi_2}\cdot\pi_1^{-1} \ge \widehat{\pi_2}\cdot\pi_1^{-1}.
\]
We will give a proof for \eqref{SplitFork} at the end of this subsection, and show first how \eqref{SplitFork} can be used to prove monadicity of $G$. Observe first that, being a split fork,
\[
\xymatrix{\hat{R}\ar@<0.5ex>[r]^{\widehat{\pi_1}}\ar@<-0.5ex>[r]_{\widehat{\pi_2}} &\hat{X}\ar[r]^{\hat{q}} &\hat{Q}}
\]
is a coequaliser diagram in $\Cat{\Tth}$ and $\CatSep{\Tth}$. Hence, there is a $\Tth$-functor $\Sup_Q:\hat{Q}\to Q$ with $\Sup_Q\cdot\hat{q}=q\cdot \Sup_X$ and $\Sup_Q\cdot\yoneda_Q=1_Q$. The situation is depicted below.
\[
\xymatrix{R\ar@<0.5ex>[r]^{\pi_1}\ar@<-0.5ex>[r]_{\pi_2}\ar[d]_{\yoneda_R} & X\ar[r]^q\ar[d]_{\yoneda_X} & Q\ar[d]^{\yoneda_Q}\ar@/^3pc/[dd]^{1_Q}\\
\hat{R}\ar@<0.5ex>[r]^{\widehat{\pi_1}}\ar@<-0.5ex>[r]_{\widehat{\pi_2}}\ar[d]_{\Sup_R} &\hat{X}\ar[r]^{\hat{q}}\ar[d]_{\Sup_X} &\hat{Q}\ar@{..>}[d]^{\Sup_Q}\\
R\ar@<0.5ex>[r]^{\pi_1}\ar@<-0.5ex>[r]_{\pi_2} & X\ar[r]^q & Q}
\]
We conclude that $Q$ is L-separated and cocomplete, and $q:X\to Q$ is cocontinuous. Next we show that
\[
\xymatrix{R\ar@<0.5ex>[r]^{\pi_1}\ar@<-0.5ex>[r]_{\pi_2} & X\ar[r]^q & Q}
\]
is indeed a coqualiser diagram in $\CocomplSep{\Tth}$. Note that
\[
\xymatrix{\hat{R}\ar@<0.5ex>[r]^{\widehat{\pi_1}}\ar@<-0.5ex>[r]_{\widehat{\pi_2}} &\hat{X}\ar[r]^{\hat{q}} &\hat{Q}}
\]
is a coequaliser diagram in $\CocomplSep{\Tth}$ since $\widehat{(-)}:\Cat{\Tth}\to\CocomplSep{\Tth}$ is left adjoint. Let $h:X\to Y$ be a cocontinuous $\Tth$-functor with cocomplete codomain such that $h\cdot\pi_1=h\cdot\pi_2$. Then there exists a cocontinuous $\Tth$-functor $f:\hat{Q}\to Y$ such that $f\cdot\hat{q}=h\cdot \Sup_X$. We consider now $f\cdot\yoneda_Q:Q\to Y$. Then 
\[
f\cdot\yoneda_Q\cdot q=f\cdot\hat{q}\cdot\yoneda_X=h\cdot \Sup_X\cdot\yoneda_X=h.
\]
Furthermore,
\begin{align*}
\Sup_Y\cdot\hat{f}\cdot\widehat{\yoneda_Q}\cdot\hat{q}
&= f\cdot \Sup_{\hat{Q}}\cdot\widehat{\yoneda_Q}\cdot\hat{q}\\
&= f\cdot\hat{q} &&\text{($\Sup_{\hat{Q}}=\mu_Q$ the multiplication of the monad $\mInj$)}\\
&= h\cdot \Sup_X\\
&= f\cdot \yoneda_Q\cdot q\cdot \Sup_X\\
&= f\cdot \yoneda_Q\cdot \Sup_Q\cdot \hat{q},
\end{align*}
and therefore $\Sup_Y\cdot\widehat{f\cdot\yoneda_Q}=f\cdot\yoneda_Q\cdot \Sup_Q$, i.e.\ $f\cdot\yoneda_Q$ is cocontinuous.
\begin{remark}
Being cocontinuous, $f\cdot\yoneda_Q$ is left adjoint. In fact, one can directly show $f\cdot\yoneda_Q\dashv q\cdot l$, where $l:Y\to X$ is right adjoint to $h:X\to Y$. To do so, let $g:Y\to\hat{Q}$ be right adjoint to $f:\hat{Q}\to Y$. Then $\yoneda_X\cdot l=q^{-1}\cdot g$, and therefore
\begin{align*}
g=\hat{q}\cdot\yoneda_X\cdot l&&\text{and}&& l=\Sup_X\cdot q^{-1}\cdot g.
\end{align*}
Hence, we have
\begin{align*}
&f\cdot\yoneda_Q\cdot q\cdot l =f\cdot\hat{q}\cdot\yoneda_X\cdot l=f\cdot g\le 1_Y\\
\intertext{and}
&q\cdot l\cdot f\cdot\yoneda_Q = q\cdot \Sup_X\cdot q^{-1}\cdot g\cdot f\cdot\yoneda_Q
\ge q\cdot \Sup_X\cdot q^{-1} \yoneda_Q=\Sup_Q\cdot\hat{q}\cdot q^{-1} \yoneda_Q=1_Q.
\end{align*}
\end{remark}

Finally, we prove \eqref{SplitFork}. Let $\pi_1,\pi_2:R\rightrightarrows X$ be an equivalence relation in $\SET$, and $q:X\to Q$ its quotient. We typically write $x\sim x'$ for $(x,x')\in R$. Furthermore, for $\frx,\frx'\in TX$ we write $\frx\sim\frx'$ whenever the pair $(\frx,\frx')$ belongs to the kernel relation of $Tq$. Since $T$ has \BC, we have
\[
\frx\sim\frx'\iff \exists\frw\in TR\,.\,(T\pi_1(\frw)=\frx)\,\&\,(T\pi_2(\frw)=\frx').
\]
Furthermore, we have to warn the reader that, when talking about an equivalence relation $\pi_1,\pi_2:R\rightrightarrows X$ in $\Cat{\Tth}$ or $\CatSep{\Tth}$, we always include that the canonical map $R\hrw X\times X$ is an embedding (and not just a monomorphism). Clearly, a sub-$\Tth$-category $R\hrw X\times X$ is an equivalence relation in $\Cat{\Tth}$ respectively in $\CatSep{\Tth}$ if and only if it is an equivalence relation in $\SET$. 
\begin{lemma}
Let $X=(X,a)$ be a L-separated $\Tth$-category and $\pi_1,\pi_2:R\rightrightarrows X$ be an equivalence relation in $\CatSep{\Tth}$. In addition, assume that $\pi_2\dashv\rho_2$\footnote{Note that, since $R$ is symmetric, $\pi_1$ is left adjoint precisely if $\pi_2$ is so.}. Then, for all $\frx,\frx'\in TX$ with $\frx\sim\frx'$ and all $x'\in X$, there exists $x\in X$ such that $x\sim x'$ and $a(\frx',x')\le a(\frx,x)$.
\end{lemma}
\begin{proof}
Since $\pi_2$ is surjective, we have $\pi_2\cdot\rho_2=1_X$. Let $\frw\in TR$ such that $T\pi_1(\frw)=\frx$ and $T\pi_2(\frw)=\frx'$. Then
\begin{align*}
a(\frx',x')&=a(T\pi_2(\frw),x')\\
&= a\times a(\frw,\rho_2(x')) &&(\rho_2(x')=(x,x') \text{ for some $x\sim x'$})\\
&=a(\frx,x)\wedge a(\frx',x'),
\end{align*}
hence $a(\frx',x')\le a(\frx,x)$.
\end{proof}
Our next goal is to describe the quotient $q:X\to Q$ of $\pi_1,\pi_2:R\rightrightarrows X$ in $\Cat{\Tth}$. In general, the quotient structure in $\Cat{\Tth}$ is difficult to handle, see \cite{Hof_Quot} for details. The situation is much better in $\Gph{\Tth}$, the category of $\Tth$-graphs and $\Tth$-graph morphisms. Here a $\Tth$-graph is a pair $(X,a)$ consisting of a set $X$ and a $\Tth$-relation $a:X\krelto X$ satisfying $e_X^\circ\le a$, and $\Tth$-graph morphisms are defined as $\Tth$-functors. Clearly, we have a full embedding $\Cat{\Tth}\hrw\Gph{\Tth}$. A surjective $\Tth$-graph morphism $f:(X,a)\to(Y,b)$ is a quotient in $\Gph{\Tth}$ if and only if $b=f\cdot a\cdot Tf^\circ$ (see also \cite{CH_TopFeat}), and the full embedding $\Cat{\Tth}\hrw\Gph{\Tth}$ reflects quotients. Furthermore, we call a $\Tth$-graph morphism (or a $\Tth$-functor) $f$ \emph{proper} if $b\cdot Tf=f\cdot a$ (see \cite{CH_EffDesc}). One easily verifies that, if $f:X\to Y$ is a proper surjection, then $f$ is a quotient in $\Gph{\Tth}$, and with $X$ also $Y$ is a $\Tth$-category.
\begin{corollary}\label{QuotProper}
Consider the same situation as in the lemma above. Let $q:X\to Q$ be the quotient of $\pi_1,\pi_2:R\rightrightarrows X$ in $\Gph{\Tth}$. Then $q$ is proper, and therefore $Q$ is a $\Tth$-category and $q:X\to Q$ is the quotient of $\pi_1,\pi_2:R\rightrightarrows X$ in $\Cat{\Tth}$.
\end{corollary}
\begin{proof}
Let $\frx\in TX$ and $y\in Q$, i.e.\ $y=q(x)$ for some $x\in X$. With $c$ denoting then structure on $Q$, we have
\[
c(Tq(\frx),y)=\bigvee\{a(\frx',x')\mid \frx'\sim\frx,\,x'\sim x\}
= \bigvee\{a(\frx,x')\mid x'\sim x\}
=\bigvee\{a(\frx,x')\mid x'\in X,\, q(x')=y\}.\qedhere
\]
\end{proof}
\begin{corollary}\label{MQuotProper}
With the same notation as above, $M(q):M(X)\to M(Q)$ is proper.
\end{corollary}
\begin{proof}
Just observe that both diagrams
\[
\xymatrix{TX\ar|-{\object@{|}}[d]_{m_X^\circ}\ar[r]^{Tq} & TQ\ar|-{\object@{|}}[d]^{m_Q^\circ}\\
TTX\ar[r]_{TTq}\ar|-{\object@{|}}[d]_{\Txi a} & TTQ\ar|-{\object@{|}}[d]^{\Txi c}\\
TX\ar[r]_{Tq} & TQ}
\]
are commutative: the upper one since $m$ has \BC, the lower one since $q$ is proper and $\Txi$ is a functor.
\end{proof}

We are now in the position to show \eqref{SplitFork}. Let $\pi_1,\pi_2:R\rightrightarrows X$ in $\CocomplSep{\Tth}$ be an equivalence relation in $\SET$. Note that $R\hrw X\times X$ is left adjoint and injective, hence a split monomorphism and therefore an embedding in $\CatSep{\Tth}$. Hence, by Corollary \ref{QuotProper}, its quotient $q:X\to Q$ in $\Cat{\Tth}$ is proper, and so is $M(q):M(X)\to M(Q)$ by Corollary \ref{MQuotProper}. Let $\psi\in\hat{X}$ and $\frx\in TX$. The structure on $X$ and $Q$ we denote as $a$ and $c$ respectively, and put $r=\Txi a\cdot m_X^\circ$ and $s=\Txi c\cdot m_Q^\circ$. We have
\begin{align*}
(q^{-1}\cdot\hat{q}(\psi))(\frx)
&=\hat{q}(\psi)(Tq(\frx))\\
&=\bigvee_{\frx'\in TX}\psi(\frx')\otimes s(Tq(\frx),Tq(\frx'))\\
&=\bigvee_{(\frx'\in TX)}\bigvee_{(\frx'':\frx''\sim\frx')}\psi(\frx')\otimes r(\frx,\frx'')\\
\intertext{and}
(\widehat{\pi_2}\cdot\pi_1^{-1}(\psi))(\frx)
&=\bigvee_{(\frx'\in TX)}\bigvee_{(\frw:T\pi_2(\frw)=\frx')} \psi(T\pi_1(\frw))\otimes r(\frx,\frx')\\
&=\bigvee_{(\frx'\in TX)}\bigvee_{(\frx'':\frx''\sim\frx')}\psi(\frx'')\otimes r(\frx,\frx').
\end{align*}
We conclude $q^{-1}\cdot\hat{q}=\widehat{\pi_2}\cdot\pi_1^{-1}$.
\begin{theorem}
The forgetful functor $G:\CocomplSep{\Tth}\to\SET$ is monadic. As a consequence, $\CocomplSep{\Tth}$ is cocomplete.
\end{theorem}
\begin{theorem}
The forgetful functor $\ForgetToV:\CocomplSep{\Tth}\to\CatSep{\V}$ is monadic.
\end{theorem}
\begin{proof}
Clearly, $\ForgetToV$ has a left adjoint and reflects isomorphisms. We show that $\ForgetToV$ preserves coequalisers of $\ForgetToV$-contractible equivalence relations (see \cite[Theorem 2.7]{MS_Monads}). Hence, let $\pi_1,\pi_2:R\rightrightarrows X$ in $\CocomplSep{\Tth}$ be a contractible equivalence relation in $\CatSep{\V}$. Then $\pi_1,\pi_2:R\rightrightarrows X$ is also an equivalence relation in $\SET$, and hence its coequaliser $q:X\to Q$ in $\SET$ underlies its coequaliser $q:X\to Q$ in $\CocomplSep{\Tth}$, moreover, $q:X\to Q$ is a proper $\Tth$-functor. Consequentely, the underlying $\V$-functor $q:X\to Q$ is proper as well, and therefore a coequaliser of $\pi_1,\pi_2:R\rightrightarrows X$ in $\CatSep{\V}$.
\end{proof}

\subsection{Densely injective $\Tth$-categories} 
Another well-known result in Topology is
\begin{theorem}
The algebras for the proper filter monad on $\TOP_0$ are precisely the T$_0$-spaces which are injective with respect to dense embeddings.
\end{theorem}
In the language of convergence, a continuous map $f:X\to Y$ is dense whenever
\[
 \forall y\in Y\,\exists\frx\in TX\,.\,Uf(\frx)\rightarrow y,
\]
and we observe that $Uf(\frx)\rightarrow y \iff \frx\; f_*\; y$. This suggests the following
\begin{definition}
A $\Tth$-module $\varphi:X\kmodto Y$ is called \emph{inhabited} if
\[
 k\le\bigwedge_{y\in Y}\bigvee_{\frx\in TX}\varphi(\frx,y).
\]
A $\Tth$-functor $f:X\to Y$ is called \emph{dense} if $f_*$ is inhabited.
\end{definition}
We hasten to remark that $f^*$ is inhabited, for each $\Tth$-functor $f:X\to Y$. Hence
\begin{proposition}\label{AdjDense}
Each left adjoint $\Tth$-functor is dense.
\end{proposition}
By definition, $\varphi:X\kmodto Y$ is inhabited if and only if $k\le \varphi\kleisli k$, where $k$ denotes the constant $\Tth$-relation $k:T1\times Z\to\V$ with value $k\in\V$, for a set $Z$. Consequentely, with $\varphi:X\kmodto Y$ and $\psi:Y\kmodto Z$ also $\psi\kleisli\varphi$ is inhabited. Furthermore, if $\varphi$ is inhabited and $\varphi\le\varphi'$, then $\varphi'$ is inhabited too. Note also that each surjective $\Tth$-functor is dense.
\begin{proposition}\label{CompCancDense}
Consider the (up to $\cong$) commutative triangle
\[
\xymatrix{X\ar[d]_f\ar[dr]^g_\cong \\ Y\ar[r]_h & Z}
\]
of $\Tth$-functors. Then the following assertions hold.
\begin{enumerate}
\item If $h$ and $f$ are dense, then so is $g$.
\item If $g$ is dense and $h$ is fully faithful, then $f$ is dense.
\item If $g$ is dense, then $h$ is dense.
\end{enumerate}
\end{proposition}
\begin{proof}
(1) is obvious since inhabited $\Tth$-modules compose. To see (2), note that from $h_*\kleisli f_*=g_*$ follows $f_*=h^*\kleisli g_*$, hence $f_*$ is inhabited and therefore $f$ is dense. (3) can be shown in a similar way.
\end{proof}
By the Yoneda Lemma (Corollary \ref{YonedaLem}), for each $\psi\in\hat{X}$ we have
\[
 \bigvee_{\frx\in TX}(\yoneda_X)_*(\frx,\psi)=\bigvee_{\frx\in TX}\psi(\frx).
\]
Hence, with
\[
 X^+=\{\psi\in\hat{X}\mid \psi\text{ is inhabited}\}
\]
and the structure being inherited from $\hat{X}$, the restriction $\yoneda_X:X\to X^+$ of the Yoneda embedding is dense. Furthermore, for a $\Tth$-module $\varphi:X\kmodto Y$ we have
\[
\varphi \text{ is inhabited}\iff \mate{\varphi}:Y\to\hat{X}\text{ factors through }X^+\hrw\hat{X}.
\]

We call a $\Tth$-category $X$ \emph{densely injective} if, for all $\Tth$-functors $f:A\to X$ and fully faithful and dense $\Tth$-functors $i:A\to B$, there exists a $\Tth$-functor $g:B\to X$ such that $g\cdot i\cong f$. A $\Tth$-category $X$ is called \emph{inhabited-cocomplete} if $X$ has all $\varphi$-weighted colimits where $\varphi$ is inhabited. Note that, when passing from
\begin{align*}
\xymatrix{A\ar[r]^f\ar@{-^{>}}|-{\object@{o}}[d]_\varphi & X\\ B} &&\text{to}&&
\xymatrix{X\ar[r]^{1_X}\ar@{-^{>}}|-{\object@{o}}[d]_{\varphi\kleisli f^*} & X,\\ B}
\end{align*}
with $\varphi$ also $\varphi\kleisli f^*$ is inhabited, so that it is enough to consider $f=1_X$ in the definition of inhabited-cocomplete. A $\Tth$-functor $f:X\to Y$ is \emph{inhabited-cocontinuous} if $f$ preserves all $\varphi$-weighted colimits where $\varphi$ is inhabited. Let $\ICocompl{\Tth}$ denote the category of inhabited-cocomplete $\Tth$-categories and inhabited-cocontinuous $\Tth$-functors between them, and $\ICocomplSep{\Tth}$ denotes its full subcategory of L-separated $\Tth$-categories. 
\begin{lemma}\label{PlusClosedinHat}
For each $\Tth$-category $X$, $X^+$ is closed under inhabited colimits in $\hat{X}$. In particular, $X^+$ is inhabited-cocomplete.
\end{lemma}
\begin{proof}
We consider the diagram 
\[
\xymatrix{X^+\ar[r]^\iota\ar@{-^{>}}|-{\object@{o}}[d]_\varphi & \hat{X},\\ Y}
\]
with $\iota:X^+\hrw\hat{X}$ being the inclusion functor and $\varphi$ inhabited. Its colimit in $\hat{X}$ is given by
\[
\yoneda_X^{-1}\cdot\mate{\varphi\kleisli{\iota^*}}:Y\to\hat{X}.
\]
Hence, for any $y\in Y$ and $\frx\in TX$,
\begin{align*}
\yoneda_X^{-1}\cdot\mate{\varphi\kleisli{\iota^*}}(y)(\frx)
=\varphi\kleisli{\iota^*}(T\yoneda_X(\frx),y)
\ge\varphi\cdot T\iota^\circ(T\yoneda_X(\frx),y)
=\varphi(T\yoneda_X(\frx),y)=\varphi\kleisli(\yoneda_X)_*(\frx,y),
\end{align*}
where in the last two expessions we consider $\yoneda_X:X\to X^+$. Since $\varphi\kleisli(\yoneda_X)_*$ is inhabited, the $\Tth$-functor $\yoneda_X^{-1}\cdot\mate{\varphi\kleisli{\iota^*}}:Y\to\hat{X}$ takes values in $X^+$ and the assertion follows.
\end{proof}
From the observations made so far it is now clear that we have the same series of results for densely injective and inhabited-cocomplete $\Tth$-categories as we proved for injective and cocomplete $\Tth$-categories.
\begin{theorem}
Let $X$ be $\Tth$-category. 
\begin{enumerate}
\item Each $\psi\in X^+$ is an inhabited colimit of representables.
\item The following assertions are equivalent.
\begin{enumerate}
\renewcommand{\theenumii}{\roman{enumii}}
\item $X$ is densely injective.
\item $\yoneda_X:X\to X^+$ has a left inverse $\Sup_X^+:X^+\to X$.
\item $\yoneda_X:X\to X^+$ has a left adjoint $\Sup_X^+:X^+\to X$.
\item $X$ is inhabited-cocomplete.
\end{enumerate}
\item Composition with $\yoneda_X:X\to X^+$ defines an equivalence
\[
 \ICocompl{\Tth}(X^+,Y)\to\Cat{\Tth}(X,Y)
\]
of ordered sets, for each inhabited-cocomplete $\Tth$-category $Y$.
\end{enumerate}
\end{theorem}
We have just seen that the inclusion functor $\ICocomplSep{\Tth}\hrw\CatSep{\Tth}$ has a left adjoint $(-)^+:\CatSep{\Tth}\to\ICocomplSep{\Tth}$. In fact, since for each $\Tth$-functor $f:X\to Y$ and each $\psi\in X^+$ we have $\hat{f}(\psi)\in Y^+$, the $\Tth$-functor $f^+:X^+\to Y^+$ is just the restriction of $\hat{f}$ to $X^+$ and $Y^+$. With a similar proof as for Proposition \ref{CharLeftAdFun} one shows
\begin{proposition}
Let $f:X\to Y$ be a $\Tth$-functor between inhabited-cocomplete $\Tth$-categories. Then the following assertions are equivalent.
\begin{eqcond}
\item $f$ is inhabited-cocontinuous.
\item We have $f\cdot \Sup^+_X\cong \Sup^+_Y\cdot\hat{f}$.
\[
 \xymatrix{X^+\ar[r]^{f^+}\ar[d]_{\Sup^+_X}\ar@{}[dr]|{\cong}
	& Y^+\ar[d]^{\Sup^+_Y}\\ X\ar[r]_f & Y}
\]
\end{eqcond}
\end{proposition}
The induced monad on $\CatSep{\Tth}$ we denote as $\mInjP=((-)^+,\yoneda,\mu)$. With the same arguments used in \ref{CocomplAlg} one verifies that $\mInjP$ is of Kock-Z\"oberlein type. We conclude
\begin{theorem}
$\ICocomplSep{\Tth}\cong(\CatSep{\Tth})^{\mInjP}$.
\end{theorem}
Finally, we consider a $\Tth$-functor $f:X\to Y$. Then $\hat{f}:\hat{X}\to\hat{Y}$ has a right adjoint $f^{-1}:\hat{Y}\to\hat{X}$ given by $f^{-1}(\psi)=\psi\kleisli f_*$. Clearly, if $f$ is dense, then $f^{-1}$ can be restricted to $f^{-1}:Y^+\to X^+$ and we have $f^+\dashv f^{-1}$. In particular, $\yoneda^+_X:X^+\to X^{++}$ is left adjoint to $\yoneda_X^{-1}:X^{++}\to X^+$, which tells us that the multiplication $\mu_X$ of $\mInjP$ is also given by $\yoneda_X^{-1}$.
\begin{proposition}
The following are equivalent for a $\Tth$-functor $f:X\to Y$.
\begin{eqcond}
\item $f$ is dense.
\item $f^+$ is left adjoint.
\item $f^+$ is dense.\newcounter{counter}\setcounter{counter}{\value{enumi}}
\end{eqcond}
If $f$ is a inhabited-cocontinuous $\Tth$-functor between inhabited cocomplete $\Tth$-categories, then any of the conditions above is equivalent to
\begin{eqcond}
\setcounter{enumi}{\value{counter}}
\item $f$ is left adjoint.
\end{eqcond}
\end{proposition}
\begin{proof}
The implication (i)$\Rw$(ii) we proved above, (ii)$\Rw$(iii) and (iv)$\Rw$(i) follow from Proposition \ref{AdjDense} and (iii)$\Rw$(i) from Proposition \ref{CompCancDense}. Finally, (ii)$\Rw$(iv) can be shown as (iii)$\Rw$(i) of Proposition \ref{CharLeftAdFun}.
\end{proof}
Finally, thanks to the considerations made above, also
\[
\xymatrix{R^+\ar@<0.5ex>[r]^{\pi_1^+}\ar@<-0.5ex>[r]_{\pi_2^+} & X^+\ar[r]^{q^+} & Q^+}
\]
is a split fork in $\CatSep{\Tth}$. Consequentely, with the same proof as in \ref{MonOverSet}, we conclude that the forgetful functor $\ICocomplSep{\Tth}\to\SET$ is monadic.
\begin{remark}
The results of this subsection suggest that in the future one should consider cocompleteness with respect to a class $\Phi$ of $\Tth$-modules, i.e.\ use \cite{KS_Colim}. Besides the classes considered in this paper, another reasonable choice is $\Phi$ being the class of all right adjoint $\Tth$-modules. In fact, this case is studied in \cite{CH_Compl,HT_LCls} where the $\Phi$-cocomplete $\Tth$-categories are called L-complete (resp.\ Cauchy-complete). Furthermore, it is easy to see that any $\Tth$-functor preserves colimits indexed by a right adjoint weight, so that the category of L-separated and $\Phi$-cocomplete $\Tth$-categories and $\Phi$-cocontinuous $\Tth$-functors is precisely the full subcategory $\CatCompl{\Tth}$ of L-complete and L-separated $\Tth$-categories of $\Cat{\Tth}$.
But be aware that, thought with the same techniques we obtain monadicity of $\CatCompl{\Tth}$ over $\CatSep{\Tth}$, the proof in \ref{MonOverSet} does not work here. The problem is that the $\Tth$-functor $q^{-1}:\hat{Q}\to\hat{X}$ does not restrict to $\tilde{Q}$ and $\tilde{X}$\footnote{Here $\tilde{X}=\{\psi\in\hat{X}\mid\text{ $\psi$ right adjoint}\}$.} since $q_*$ is in general not right adjoint. This is not a surprise, since, for instance, any ordered set is L-complete, hence the category of L-complete and L-separated ordered set coincides with the category of anti-symmetric ordered sets (and monotone maps). The canonical forgetful functor from this category to $\SET$ is surely not monadic. Furthermore, the canonical forgetful functor from the category of L-complete and L-separated topological spaces (= sober spaces) and continuous maps to $\SET$ is also not monadic. 
\end{remark}

%\bibliographystyle{halpha}
%\bibliography{/home/dirk/Documents/tex/bibliography/bibliography}

\begin{thebibliography}{BCSW83}

\bibitem[AHS90]{AHS}
Ji{\v{r}}{\'{\i}} Ad{\'a}mek, Horst Herrlich, and George~E. Strecker.
\newblock {\em Abstract and concrete categories}.
\newblock Pure and Applied Mathematics (New York). John Wiley \& Sons Inc., New
  York, 1990.
\newblock The joy of cats, A Wiley-Interscience Publication.

\bibitem[Bar70]{Bar_RelAlg}
Michael Barr.
\newblock Relational algebras.
\newblock In {\em Reports of the Midwest Category Seminar, IV}, pages 39--55.
  Lecture Notes in Mathematics, Vol. 137. Springer, Berlin, 1970.

\bibitem[BCSW83]{BCRW_VarEnr}
Renato Betti, Aurelio Carboni, Ross Street, and Robert Walters.
\newblock Variation through enrichment.
\newblock {\em J. Pure Appl. Algebra}, 29(2):109--127, 1983.

\bibitem[Ben63]{Ben:63}
J.~Benabou.
\newblock {Categories avec multiplication.}
\newblock {\em C. R. Acad. Sci., Paris}, 256:1887--1890, 1963.

\bibitem[Ben65]{Ben:65}
J.~Benabou.
\newblock {Categories r\'elatives.}
\newblock {\em C. R. Acad. Sci., Paris}, 260:3824--3827, 1965.

\bibitem[Bur71]{Bur_TCat}
Albert Burroni.
\newblock {$T$}-cat\'egories (cat\'egories dans un triple).
\newblock {\em Cahiers Topologie G\'eom. Diff\'erentielle}, 12:215--321, 1971.

\bibitem[BvBR98]{BBR_GenMet}
M.~M. Bonsangue, F.~van Breugel, and J.~J. M.~M. Rutten.
\newblock Generalized metric spaces: completion, topology, and powerdomains via
  the {Y}oneda embedding.
\newblock {\em Theoret. Comput. Sci.}, 193(1-2):1--51, 1998.

\bibitem[CH03]{CH_TopFeat}
Maria~Manuel Clementino and Dirk Hofmann.
\newblock Topological features of lax algebras.
\newblock {\em Appl. Categ. Structures}, 11(3):267--286, 2003.

\bibitem[CH04]{CH_EffDesc}
Maria~Manuel Clementino and Dirk Hofmann.
\newblock Effective descent morphisms in categories of lax algebras.
\newblock {\em Appl. Categ. Structures}, 12(5-6):413--425, 2004.

\bibitem[CH07]{CH_Compl}
Maria~Manuel Clementino and Dirk Hofmann.
\newblock {L}awvere completeness in {T}opology.
\newblock Technical report, University of Aveiro, 2007,
  arXiv:math.CT/0704.3976.

\bibitem[CHT04]{CHT_OneSetting}
Maria~Manuel Clementino, Dirk Hofmann, and Walter Tholen.
\newblock One setting for all: metric, topology, uniformity, approach
  structure.
\newblock {\em Appl. Categ. Structures}, 12(2):127--154, 2004.

\bibitem[CT03]{CT_MultiCat}
Maria~Manuel Clementino and Walter Tholen.
\newblock Metric, topology and multicategory---a common approach.
\newblock {\em J. Pure Appl. Algebra}, 179(1-2):13--47, 2003.

\bibitem[Day75]{Day_Filter}
Alan Day.
\newblock {Filter monads, continuous lattices and closure systems.}
\newblock {\em Can. J. Math.}, 27:50--59, 1975.

\bibitem[EK66]{EK_CloCat}
Samuel Eilenberg and G.~Max Kelly.
\newblock Closed categories.
\newblock In {\em Proc. Conf. Categorical Algebra (La Jolla, Calif., 1965)},
  pages 421--562. Springer, New York, 1966.

\bibitem[Esc97]{Esc_InjSp}
Mart{\'{\i}}n~H{\"o}tzel Escard{\'o}.
\newblock Injective spaces via the filter monad.
\newblock In {\em Proceedings of the 12th Summer Conference on General Topology
  and its Applications (North Bay, ON, 1997)}, volume~22, pages 97--100, 1997.

\bibitem[Hof05]{Hof_Quot}
Dirk Hofmann.
\newblock An algebraic description of regular epimorphisms in topology.
\newblock {\em J. Pure Appl. Algebra}, 199(1-3):71--86, 2005.

\bibitem[Hof07]{Hof_TopTh}
Dirk Hofmann.
\newblock {Topological theories and closed objects.}
\newblock {\em Adv. Math.}, 215(2):789--824, 2007.

\bibitem[HT08]{HT_LCls}
Dirk Hofmann and Walter Tholen.
\newblock {L}awvere completion and separation via closure.
\newblock Technical report, University of Aveiro, 2008,
  arXiv:math.CT/0801.0199.

\bibitem[Kel82]{Kel_EnrCat}
Gregory~Maxwell Kelly.
\newblock {\em Basic concepts of enriched category theory}, volume~64 of {\em
  London Mathematical Society Lecture Note Series}.
\newblock Cambridge University Press, Cambridge, 1982.

\bibitem[Koc95]{Koc_MonAd}
Anders Kock.
\newblock {Monads for which structures are adjoint to units.}
\newblock {\em J. Pure Appl. Algebra}, 104(1):41--59, 1995.

\bibitem[KS05]{KS_Colim}
Gregory~Maxwell Kelly and Vincent Schmitt.
\newblock {Notes on enriched categories with colimits of some class.}
\newblock {\em Theory Appl. Categ.}, 14:399--423, 2005.

\bibitem[Law73]{Law_MetLogClo}
F.~William Lawvere.
\newblock Metric spaces, generalized logic, and closed categories.
\newblock {\em Rend. Sem. Mat. Fis. Milano}, 43:135--166 (1974), 1973.
\newblock Also in: Repr. Theory Appl. Categ. 1:1--37 (electronic), 2002.

\bibitem[Low97]{Low_ApBook}
Robert Lowen.
\newblock {\em Approach spaces}.
\newblock Oxford Mathematical Monographs. The Clarendon Press Oxford University
  Press, New York, 1997.
\newblock The missing link in the topology-uniformity-metric triad, Oxford
  Science Publications.

\bibitem[MS04]{MS_Monads}
John MacDonald and Manuela Sobral.
\newblock Aspects of monads.
\newblock In {\em Categorical foundations}, volume~97 of {\em Encyclopedia
  Math. Appl.}, pages 213--268. Cambridge Univ. Press, Cambridge, 2004.

\bibitem[Sco72]{Sco_ContLat}
Dana Scott.
\newblock Continuous lattices.
\newblock In {\em Toposes, algebraic geometry and logic (Conf., Dalhousie
  Univ., Halifax, N. S., 1971)}, pages 97--136. Lecture Notes in Math., Vol.
  274. Springer, Berlin, 1972.

\bibitem[Wag94]{Wag_PhD}
Kim~Ritter Wagner.
\newblock {\em Solving Recursive Domain Equations with Enriched Categories}.
\newblock PhD thesis, Carnegie Mellon University, 1994.

\bibitem[Was02]{Was_PhD}
Pawel Waszkiewicz.
\newblock {\em Quantitative Continuous Domains}.
\newblock PhD thesis, School of Computer Science, The University of Birmingham,
  Edgbaston, Birmingham B15 2TT, May 2002.

\bibitem[Woo04]{Woo_OrdAdj}
R.J. Wood.
\newblock Ordered sets via adjunctions.
\newblock In {\em Categorical foundations}, volume~97 of {\em Encyclopedia
  Math. Appl.}, pages 5--47. Cambridge Univ. Press, Cambridge, 2004.

\end{thebibliography}
\end{document}